\documentclass[11pt]{amsart}

\usepackage{palatino}
\usepackage{graphicx}
\usepackage{geometry}
\usepackage{hyperref}
\usepackage{soul,color}
\usepackage{enumerate}
\usepackage{paralist, amsthm,amsfonts,amsmath,amssymb,graphics}
\usepackage{enumitem}
\usepackage{color}
\definecolor{DarkGreen}{RGB}{0,190,0}
\setlist{nosep}
\numberwithin{equation}{section}

\usepackage[stretch=10]{microtype}
\usepackage{mathrsfs}
\newcommand{\footnotetextplain}[1]{\begingroup\def\@thefnmark{}%
  \long\def\@makefntext##1{\parindent 0pt\noindent ##1}\@footnotetext{#1}
  \endgroup}

\makeatletter 
\newcommand{\TeoremaAmbFinalMarcat}[1]{%
  \expandafter\gdef\csname end#1\endcsname{\@endtheorem}}
  
               {\hfill\rule{2.5mm}{2.5mm} \vspace{\parskip} } 
\newtheorem{theorem}{Theorem}

\newtheorem{proposition}{Proposition}[section]
\newtheorem{corollary}[proposition]{Corollary}
\newtheorem{lemma}[proposition]{Lemma}

\newtheorem{claim}{Claim}
\newtheorem{conj}{Conjecture} 

\theoremstyle{definition}
\newtheorem{definition}[proposition]{Definition} \TeoremaAmbFinalMarcat{defi}
 \TeoremaAmbFinalMarcat{ex}
\newtheorem{remark}[proposition]{Remark} \TeoremaAmbFinalMarcat{rem}
 \TeoremaAmbFinalMarcat{conv}

\def\@enum@{\list{\csname label\@enumctr\endcsname}%
           {\usecounter{\@enumctr}\def\makelabel##1{\hss\llap{##1}}
           \itemsep=2pt\parsep=0pt\topsep=3pt plus 1pt minus 1 pt}}

\newcommand{\xset}[1]{\left\{ #1 \right\}}

\newcommand{\xnorm}[1]{ \Vert #1 \Vert }

\newcommand{\zz}[1]{\mathbb #1}

\newcommand{\D}{\mathcal{D}}

\title{Local geometry of random
geodesics on negatively curved surfaces}

\author{Jayadev S.~Athreya}
\author{Steven P.~ Lalley}
\author{Jenya Sapir}
\author{Matt Wroten}
\email{jathreya@uw.edu}
\email{lalley@galton.uchicago.edu}
\email{sapir@math.binghamton.edu}
\email{mwroten@cshl.edu}
\address{Department of Mathematics, University of Washington, PO Box 354350, Seattle, WA, 98195-4350}
\address{Department of Statistics, University of Chicago, 5734 University Avenue, Chicago, IL 60637}
\address{Department of Mathematical Sciences, Binghamton University, PO Box 6000 Binghamton, New York 13902-6000}
\address{Cold Spring Harbor Laboratory, One Bungtown Road Cold Spring Harbor, NY 11724}

\thanks{Athreya is supported by National Science Foundation CAREER
grant DMS 1559860.}
\thanks{Lalley is supported by National Science Foundation Award  DMS 1612979.}

\begin{document}
\maketitle

\begin{abstract}
We show that the tessellation of a compact, negatively curved surface
induced by a  long random geodesic segment, when properly scaled, 
looks locally like a \emph{Poisson line process}. This implies that
the global statistics of the tessellation -- for instance, the
fraction of triangles -- approach those of the limiting Poisson line
process.
\end{abstract}

\section{Main Results: Intersection Statistics of Random
Geodesics}\label{sec:local}

\subsection{Local Statistics}\label{ssec:local} Any  sufficiently long
geodesic segment $\gamma$ on a compact, negatively curved surface $S$
partitions $S$ into a finite number of non-overlapping geodesic
polygons of various shapes and sizes, whose vertices\footnote{A long
  segment of a random geodesic ray doesn't quite induce a
  tessellation, as there will be two faces [triangles, quadrilaterals,
  or whatever] that contain the two ends of the geodesic segment. We
  ignore these, however, since they will not influence statistics when
  the length of the geodesic segment is large.}  are the
self-intersection points of $\gamma$. If a geodesic segment $\gamma$
of length $T$ is chosen by selecting its initial tangent vector $v$ at
random, according to (normalized) Liouville measure $\mu_{L}$ on the unit
tangent bundle $T^1S$, then with probability $1$, as
$T \rightarrow \infty$ the maximal diameter of a polygon in the
induced partition will converge to $0$, and hence the number of
polygons in the partition will become large. The goal of this paper is
to elucidate some of the statistical properties of this random
polygonal partition for large $T$. Our main result will be a
\emph{local} geometric description of the partition: roughly, this
will assert that in a neighborhood of any point $x\in S$ the partition
will, in the large$-T$ limit, look as if it were induced by a
\emph{Poisson line process} \cite{miles:pnas1}, \cite{miles:pnas2}. We
will also show that this result has implications for the \emph{global}
statistics of the partition: for instance, it will imply that with
probability $\approx 1$ the fraction of polygons in the partition that
are \emph{triangles} will stabilize near a non-random limiting value
$\tau_{3}>0$.

\begin{definition}\label{definition:plp}
A \emph{Poisson line process} $\mathcal{L}$ of intensity $\kappa >0$
is a random collection $\mathcal{L}=\{L_{n} \}_{n\in
\zz{Z}}$ of lines in $\zz{R}^{2}$ constructed as follows.  Let $\{(R_{n},\Theta_{n})
\}_{n\in \zz{Z}}$ be the points of a Poisson point
process\footnote{The ordering of the points doesn't really matter, but
for definiteness take $\dotsb <R_{-1}<0<R_{0}<R_{1}<\dotsb$. The
assumption that $\{(R_{n},\Theta_{n})
\}_{n\in \zz{Z}}$ is a Poisson point process of intensity $\kappa
/\pi$ is equivalent to the assumption that $\{R_{n} \}_{n\in \zz{R}}$
is a Poisson point process of intensity $\kappa$ on $\zz{R}$ and that
$\{\Theta_{n} \}_{n\in \zz{Z}}$ is an independent sequence of
i.i.d. random variables with uniform distribution on $[0,\pi]$.}
of intensity $\kappa /\pi  $ on the infinite strip $\zz{R}\times [0,\pi )$.
For each $n\in \zz{Z}$ let $L_{n}$ be
the line 
\begin{equation}\label{eq:point-to-line-mapping}
	L_{n}:=\{(x,y)\in \zz{R}^{2}\,:\, R_{n}=x \cos \Theta_{n}+y\sin \Theta_{n}\}.
\end{equation}
\end{definition}

That is, we consider the line through the origin of angle $\Theta_n$ to the horizontal, and $L_n$ is the line orthogonal to this line passing through it at distance $R_n$ from the origin. Observe that the mapping \eqref{eq:point-to-line-mapping} of points
$(r,\theta)$ to lines is a bijection from the strip $\zz{R}\times
[0,\pi )$ to the space of all lines in $\zz{R}^{2}$. 
For any convex region $\Omega \subset \zz{R}^{2}$, call the restriction to
$\Omega$ of a Poisson line process a Poisson line process in $\Omega$.
It is not difficult to show (see Lemma~\ref{lemma:local-sparsity}
below) that, with probability one, if $\Omega$ is a bounded domain
with piecewise smooth boundary then the Poisson line process in
$\Omega$ will consist of only finitely many line segments, and that at
most two line segments will intersect at any point of $\Omega$. For
any realization of the process, the line segments will uniquely
determine (and be determined by) their intersection points with
$\partial \Omega$, grouped in (unordered) pairs.

In order to formulate our main result, we must explain how geodesic
segments in a small neighborhood of a point $x \in S$ are associated
with line segments in the tangent space $T_{x}S$.  We shall assume
throughout that the Riemannian metric $\varrho$ on $S$ is
$C^{\infty}$; therefore, geodesics are $C^{\infty}$ curves that depend
smoothly on their initial tangent vectors. Furthermore, we will only
consider geodesics  of unit speed. Fix $x \in S$, and
consider a small disk $D (x,r)$ on $S$ of radius $r$ centered at
$x$. A (unit-speed) geodesic ray $\gamma_{t}(v) $ with initial tangent
vector $v\in T^{1}S$ distributed according to normalized Liouville
measure $\mu_{L}$ (that is, $v=(y,\theta)$ where $y\in S$ is
distributed according to normalized surface area measure and $\theta$
is distributed according to the uniform distribution on the set
$[0, 2\pi]$ of directions based at $x$) will, with probability one,
eventually enter $D (x,r)$, at a time roughly of order $1/r$ (this
will follow from our main results). Thus, if we wish to study the
local intersection statistics of a random geodesic segment of (large)
length $T$ in a neighborhood of $x$, we should focus on the
intersections of the geodesic segment  with neighborhoods
of $x$ of diameters proportional to $1/T$.

For any $\alpha>0$ and $T>0$, set $\D_T(x,\alpha) : = D(x, \alpha T^{-1})$ to be the ball of radius $\alpha/T$ about $x$ in $S$. Let $\exp_{x}:T_{x}S \rightarrow S$ be the exponential mapping. Then,
\begin{equation}
  \label{eq:alphaNH}
  \D_T(x,\alpha) = \xset {\exp_{x}(v)\,:\, \xnorm{v}\leq
    \alpha T^{-1}}.
\end{equation}
The boundary $\partial (x, \alpha T^{-1})$ is a smooth closed curve. Consequently, the intersection of $\D_T(x, \alpha)$ with any geodesic segment will
consist of (i) finitely many geodesic crossings of
$\mathcal{D}_{T}(x,\alpha)$; (ii) up to two incomplete geodesic
crossings; and (iii) a finite number of isolated points on
$\partial \mathcal{D}_{T}(x,\alpha)$, the latter coming from
tangencies of the geodesic with the boundary. Since the set of all
unit tangent vectors tangent to the curve
$\partial \mathcal{D}_{T}(x,\alpha)$ has Liouville measure $0$,
tangent intersections will have probability zero if the initial vector
of the geodesic is chosen randomly; hence, we shall henceforth ignore
these. Furthermore, incomplete geodesic crossings will occur if and
only if the initial or terminal point of the geodesic segments lies in
the interior of $D(x, \alpha T^{-1})$; this will occur with
probability of order
$O(\textrm {area} (D(x, \alpha T^{-1}))=O(T^{-2})$, and so
can also be ignored in the $T\rightarrow \infty$ limit. Thus, with
probability $\rightarrow 1$, the intersection  consists of
finitely many geodesic crossings. Now any
geodesic crossing of $D(x, \alpha T^{-1})$ pulls back, via
the scaled exponential mapping $v \mapsto \exp_{x}(v/T)$, to a smooth
curve in the ball $B(0, \alpha)$ with endpoints on the circle
$\partial B(0, \alpha)$. When $T$ is large, such a curve will closely
approximate the \emph{chord} of the circle with the same endpoints on
$\partial B(0, \alpha)$. 

Suppose now that $v\in T^{1}S$ is a unit vector chosen randomly
according to normalized Liouville measure $\mu_{L}$. Define
$\mathcal{I}_{T}=\mathcal{I}_{T}(v; x, \alpha)$ to be the intersection
of the geodesic segment $\gamma_{[0,T]}(v)$ with the set
$D(x,\alpha T^{-1})$, and define
$\mathcal{L}_{T}= \mathcal{L}_{T}(v; x, \alpha)$ to be the finite set
of chords in $B(0,\alpha)\subset T_{x}S$ obtained by pulling back the
geodesic crossings from $\mathcal{I}_{T}$ and then replacing the
resulting curves by the corresponding chords.

\begin{theorem}\label{theorem:random-local}
  Let $S$ be a compact surface of genus $g \geq 2$, and assume that $S$
  is endowed with a $C^{\infty}$ Riemannian metric $\varrho$ of negative
  curvature.  Fix $x\in S$ and $\alpha >0$, and let $\mathcal{L}_{T}=
  \mathcal{L}_{T}(v; x,\alpha)$ be the random chord process
  corresponding to the
  intersection $\mathcal{I}_{T}=\mathcal{I}_{T}(v; x , \alpha)$ of a
  random geodesic segment (i.e., one whose initial tangent vector $v$
  is chosen randomly according to the normalized Liouville
  measure) of length $T$ with the neighborhood  $D(x;
  \alpha T^{-1})$ of $x$.  As $T \rightarrow \infty$,
  the random chord process $\mathcal{L}_{T}$ converges in
  distribution to a Poisson line process in $B (0;\alpha)$ of
  intensity
  \begin{equation}
    \label{eq:kappa-S}
    \kappa=\kappa_{S}=\frac{1}{\,\text{area} (S)}.
  \end{equation}
\end{theorem}

Because the elements of the random processes here live in somewhat
unusual spaces (finite unions of chords), we now elaborate
on the meaning of convergence in distribution. In general, we say that
a sequence of random elements of a complete metric space $\mathcal{X}$
converge in distribution if their distributions (the induced
probability measures on $\mathcal{X}$) converge weakly.  Weak
convergence is defined as follows \cite{billingsley}: if
$\mu_{n},\mu$ are Borel probability measures on a complete metric
space $\mathcal{X}$, then $\mu_{n}\rightarrow \mu$ weakly if for every
bounded, continuous function $f:\mathcal{X}\rightarrow \zz{R}$,
\begin{equation}\label{eq:weak-conv-definition}
  \lim_{n \rightarrow\infty}\int f\, d\mu_{n} =\int  f \,d\mu.
\end{equation}
In Theorem~\ref{theorem:random-local}, the appropriate metric space is 
\begin{displaymath}
  \mathcal{X}=\cup_{n=0}^{\infty}\mathcal{X}_{n} 
\end{displaymath}
  where
$\mathcal{X}_{n}$ is the set of all collections
of $n$ \emph{unordered}
pairs $y_{i}, z_{i}\in \partial B (0;\alpha)$. For any two such unordered
pairs $\xset {y,z} , \xset {y',z'}$, set
\begin{displaymath}
  d(\xset {y,z} , \xset {y',z'})=\min (d(y,y')+d(z,z'),
  d(y,z')+d(z,y')); 
\end{displaymath}
and for any two
elements $F,F'\in \mathcal{X}$, define
\begin{align*}
  d(F,F')&= \min _{\pi \in \mathcal{S}_{n}} d\left (\xset {y_{i},z_{i}},\xset
  {y'_{ \pi (i)},z'_{\pi (i)}}\right ) \quad
  \textrm {if} \; F,F' \in \mathcal{X}_{n}, \\
  &= \infty \quad \textrm {otherwise}
\end{align*}
where $\mathcal{S}_{n}$ is the group of permutations of the set $[n]$.
Henceforth, we will refer to this space $\mathcal{X}$ as
\emph{configuration space} (the dependence on the parameter
$\alpha >0$ will be suppressed).

The proof of Theorem~\ref{theorem:random-local}  will also show that
the limiting Poisson line processes in neighborhoods of distinct
points of $S$ are independent.

\begin{theorem}\label{theorem:random-local-pairs}
  Fix two distinct points $x,x'\in S$ and $\alpha >0$, and let
  $\mathcal{L}_{T}$ and $\mathcal{L}'_{T}$ be the chord processes
  induced by intersections of a random geodesic of length $T$ with the
  neighborhoods $D(x;\alpha T^{-1})$ and
  $D(x';\alpha T^{-1})$, respectively. Then as
  $T \rightarrow \infty$, the random chord processes $\mathcal{L}_{T}$
  and $\mathcal{L}'_{T}$ converge jointly in distribution to a pair of
  independent Poisson line processes in $B (0;\alpha)$, both of
  intensity $\kappa = \frac{1}{\text{area}(S)}$, as in \eqref{eq:kappa-S}.
\end{theorem}

\subsection{Heuristics}\label{ssec:heuristics}
There is an explanation for the convergence to Poisson line
processes that falls short of being a complete
proof. This heuristic argument is, in essence, the same as that used
in \cite{lalley:si1} to guess the limiting frequency of
self-intersections of a random geodesic segment. It rests on the fact
that the (normalized) Liouville measure $\mu_{L}$ on the unit tangent
bundle $T^{1}S$ is a mixing invariant measure for the geodesic flow.

Let $\tilde{\gamma} : [0,\infty) \rightarrow T^{1}S$ be a random
geodesic ray with distribution $\mu_{L}$, viewed as a (random) curve
in the unit tangent bundle $T^{1}S$, and let $\gamma$ be its
projection to the surface $S$. Since $\mu_{L}$ is invariant for the
geodesic flow, for any fixed time $t>0$ the random point $\gamma(t)$
will be uniformly distributed on $S$ (according to normalized surface
area measure), and the tangent angle ${\gamma}'(t)$ will be
uniformly distributed on $[0, 2\pi]$ (according to normalized Lebesgue
measure). Fix $\varepsilon >0$, and let $T=N\varepsilon$ be a large
integer multiple of $\varepsilon$; then the geodesic segment
$\tilde{\gamma}[0,T)$ can be partitioned into $N$ nonoverlapping
segments $\Gamma_{j}:=\tilde{\gamma}[j\varepsilon, (j+1)\varepsilon)$,
each of whose initial tangent vectors $\tilde{\gamma}(j\varepsilon )$
is uniformly distributed according to $\mu_{L}$. If $\varepsilon$ is
sufficiently small then any pair of segments $\Gamma_{j}, \Gamma_{j'}$
will intersect at most once. Moreover, as $\varepsilon \rightarrow 0$
the segments $\Gamma_{j}$ approximate straight line segments of length
$\varepsilon $ in the tangent plane at the initial point.

Now we appeal to the fact that the geodesic flow is mixing relative to
$\mu_{L}$. This implies that for any two integers $j,j'$ such that
$|j-j'|$ is large, the random vectors $\tilde{\gamma}(j \varepsilon)$ and
$\tilde{\gamma}(j' \varepsilon)$ of the random segments $\Gamma_{j}$ and
$\Gamma_{j'}$ are \emph{approximately} independent.  This suggests that the
pattern and number of self-intersections in $\gamma[0,T]$ should not
differ appreciably from those of a random sample of $N$ independent
random geodesic segments $\Gamma_{j}'$ of length $\varepsilon$, each
of whose initial tangent vectors is randomly chosen from $\mu_{L}$.

\begin{figure}[h!]
 \centering 
 \includegraphics[scale=.8]{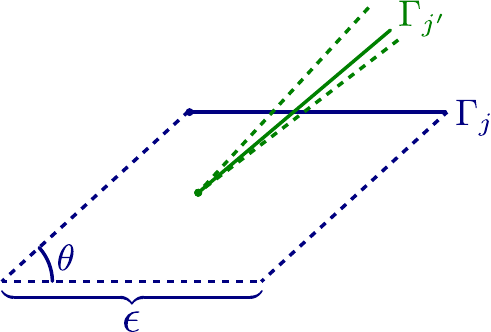}
 \caption{$\Gamma_j$ and $\Gamma_{j'}$ intersect in an angle close to $\theta$.}
 \label{fig:Rhombus}
\end{figure}

Consider, in particular, the number of self-intersections of
$\gamma[0,T)$.  For $\theta \in [0, \pi]$ and for any pair of indices
$j,j'$, the event $F_{j,j'}(\theta; d\theta)$ that the projections to
$S$ of the segments $\Gamma_{j}$ and $\Gamma_{j'}$ cross at angle
between 
$\theta- \theta$ and $\theta+ d\theta$ is, up to an error of size $O(d\theta)$, the
same as the event that (i) the point $\gamma(j' \varepsilon) $ lies in
a rhombus on $S$ whose sides meet at angle $\theta$ and whose ``top''
side is the projection to $S$ of $\Gamma_{j}$, and (ii) the tangent
angle of $\tilde{\gamma}(j' \varepsilon)$ differs from that of
$\tilde{j \varepsilon}$ by $\theta \pm d \theta$ (Figure
\ref{fig:Rhombus}). (Similarly, for $\Gamma_{j}$ and $\Gamma_{j'}$ to
cross at angle $\pi+ \theta \pm d\theta$ the ``bottom'' side of the
rhombus should be the projection of $\Gamma_{j}$.)  Since $\Gamma_{j}$
and $\Gamma_{j'}$ are approximately independent, the probability of
this event is (approximately) the relative area of this rhombus times
$2 d\theta$ divided by $2\pi$. Summing  over $\theta$ and taking $d
\theta \rightarrow 0$ now shows
that the probability of intersection is about
\begin{displaymath}
  \frac{ 2 \varepsilon^{2} \int_{0}^{\pi} \sin \theta \, d\theta }{2
    \pi \textrm {area}(S)}
  =  \frac{2 \varepsilon ^{2}\kappa}{\pi }
\end{displaymath}
 where
$\kappa=(1/ \textrm {area}(S))$. This fails, of course, if $|j-j'|$ is
small, but for most pairs $j,j'$ the difference will be
large. Consequently, by the law of large numbers, the number of
self-intersections of the segment $\gamma[0,T]$, when divided by
$T^{2}$, should satisfy
\begin{displaymath}
  \frac{1}{T^{2}} \sum_{j=1}^{N}\sum_{j'=j+1}^{N}\mathbf{1} \xset
  {\Gamma_{j'} \textrm { crosses } \Gamma_{j}} \approx
  \frac{\kappa}{\pi}.
\end{displaymath}

Amplification of this argument ``explains'' the local convergence of
the induced tessellation to the Poisson line process. Consider, for
instance, the number of distinct geodesic arcs that cross the disk
$D(x, \alpha/T)$, for some fixed point $x \in S$: we will argue that
this should have approximately a Poisson distribution.  Choose
$\varepsilon$ small, and let $T=N \varepsilon$ be an integer multiple
of $\varepsilon $ large enough that $1/T \ll \varepsilon$. For each
$j\leq N$, the probability that (the projection of) $\Gamma_{j}$
crosses the disk $D(x, \alpha T^{-1})$ is, for small $\varepsilon$,
about $C\varepsilon \alpha /T=C \alpha /N$ for a suitable geometric
constant $C>0$. (This follows by a simple geometric argument similar
to that given above for self-intersections.) Thus, if the random
segments $\Gamma_{j}$ were actually independent, the number that would
cross the disk $D(x, \alpha /T)$ would be the sum of $N$ independent
Bernoulli random variables each with mean $C \alpha /N$. For $N$
large, the distribution of this count would therefore converge to
Poisson with mean $C \alpha$ (cf. Proposition~\ref{proposition:lsn}
below).

The sticky point, of course, is that the random segments $\Gamma_{j}$
are not independent. What is worse, the events of interest (for
instance, the event that $\Gamma_{j}$ crosses the disk
$D(x,\alpha T^{-1})$) are events whose probabilities become small as
$N$ becomes large; thus, the mixing property of the geodesic flow does
not by itself imply that
\begin{displaymath}
  \frac{P (\xset {\Gamma_{j} \textrm { crosses }
      D(x,\alpha T^{-1})}\cap \xset {\Gamma_{j'} \textrm { crosses }
      D(x,\alpha T^{-1})})}{P \xset {\Gamma_{j} \textrm { crosses }
      D(x,\alpha T^{-1})}P \xset {\Gamma_{j'} \textrm { crosses }
      D(x,\alpha T^{-1})}} \approx 1
\end{displaymath}
even for $|j-j'| $ large. The rigorous arguments to be given below are
largely designed to circumvent the failure of mixing at this level by
exploiting the Gibbsean structure of the Liouville measure.

Mixing problems in which the events of interest have probabilities
tending to zero are known as ``shrinking target'' problems.  Such
problems occur naturally in hyperbolic dynamics: see, for instance,
~\cite{Sullivan}, where the ``target'' is one of the cusps of a
non-compact hyperbolic surface of finite area, or
Kleinbock-Margulis~\cite{KM}, who consider related problems for
diagonal flows on finite-volume homogeneous spaces.  For shrinking
target problems where the targets lie in the compact part of the
space, see Dolgopyat~\cite{Dolgopyat2}, and
Maucourant~\cite{Maucourant}. Unfortunately none of these results is
easily adapted to the problems we consider here.

\subsection{Global Statistics}\label{ssec:global}

Theorems \ref{theorem:random-local}--\ref{theorem:random-local-pairs}
describe the ``local'' structure of the random tessellation
$\mathcal{T}_{T}$ of the surface $S$ induced by a long segment
$\gamma[0,T]$ of a random geodesic. The tessellation $\mathcal{T}_{T}$
will consist of geodesic polygons, typically of diameter of order $T^{-1}$,
since the $O (T^{2})$ self-intersections will subdivide the length $T$
geodesic segment into sub-segments of length $O (T^{-1})$. Thus, it is
natural to look at the statistics of the scaled tessellation
$T\mathcal{T}_{T}$, which we  view as consisting of a random
number of triangles, quadrilaterals, etc., each with its own set of
side-lengths and interior angles.

The empirical frequencies of triangles,
quadrilaterals, etc. and the empirical distribution of side-length and
interior-angle sets in a Poisson line process of intensity $\kappa$
on the ball $B (0; \alpha)$ of radius $\alpha$ converge as $\alpha
\rightarrow \infty$.  (These results are evidently due to R. E. Miles~\cite{miles:pnas1}, \cite{miles:pnas2};
proofs are given in section~\ref{sec:preliminaries} below.)
Theorem~\ref{theorem:random-local}  asserts that  when $L$ is large,
then for any point $x\in S$ the statistics of the polygonal partition
in $B (x;\alpha^{-1}L)$ induced by a random geodesic segment of length
$L$ should approach those of a Poisson line process. From this
observation we will deduce the following assertion regarding
\emph{global} statistics.

\begin{theorem}\label{theorem:global-stats}
  Let $\mathcal{T}_{T}$ be the tessellation of $S$ induced by a random
  geodesic of length $T$. Then with probability approaching $1$ as
  $T \rightarrow \infty$, the empirical frequencies of triangles,
  quadrilaterals, etc. and the empirical distribution of side-length
  and interior-angle sets in $\mathcal{T}_{T}$ approach the
  corresponding theoretical frequencies for a Poisson line process.
\end{theorem}

For example, for each $v \in T^1(S)$, let $\mathcal T_T(v)$ be the tessellation induced by the length $T$ arc with initial direction $v$. Let $f_3$ be the function that returns the frequency of triangles in a tessellation. Suppose the expected value of $f_3$ is $\tau_3$ for a Poisson line process. Then we show that for any $\epsilon > 0$,
\[
 \lim_{T \to \infty} \mu_L \{v \in T^1(S) \,:\, |f_3(\mathcal T_T(v)) - \tau_3| > \epsilon\} = 0
\]
where $\mu_L$ is the Liouville measure on $T^1(S)$.

\textbf{Plan of the paper.} The proofs of
Theorems~\ref{theorem:random-local}--\ref{theorem:random-local-pairs}
will occupy most of the paper. The strategy will be to reduce the
problem to a corresponding counting problem in symbolic
dynamics. Preliminaries on Poisson line processes will be collected in
section~\ref{sec:preliminaries}, and preliminaries on symbolic
dynamics for the geodesic flow in section~\ref{sec:symbolicDynamics}.
Section~\ref{sec:heuristics} will be devoted to heuristics and a
reformulation of the problem; the proofs of
Theorems~\ref{theorem:random-local}--\ref{theorem:random-local-pairs}
will then be carried out in sections~
\ref{sec:no-prefixes} - \ref{sec:final-step}. 
Theorem~\ref{theorem:global-stats} will be proved in
section~\ref{sec:global-stats}. Finally, in section~\ref{sec:ext}, we
give a short list of conjectures, questions, and possible 
extensions of our main results.

\section{Preliminaries: Poisson line
processes}\label{sec:preliminaries}

The Poisson line process and its generalizations have a voluminous
literature, with notable early contributions by
Miles~\cite{miles:pnas1}, \cite{miles:pnas2}. See \cite{kendall-et-al}
for an extended discussion and further pointers to the literature. In
this section we will record some basic facts about these
processes. These are mostly known -- some of them are stated as
theorems in \cite{miles:pnas1} without proofs -- but proofs are not easy to
track down, so we shall provide proof sketches in
Appendix~\ref{sec:appendix}.

\subsection{Statistics of a Poisson line process}\label{ssec:plp}

\begin{lemma}\label{lemma:translation-invariance}
A Poisson line process of constant intensity $\kappa$ is both
rotationally and translationally invariant, that is, if $A$ is any
isometry of $\zz{R}^{2}$ then the configuration $\{AL_{n} \}_{n\in
\zz{Z}}$ has the same joint distribution as the configuration $\{L_{n}
\}_{n\in \zz{Z}}$. 
\end{lemma}

\begin{remark}\label{remark:ti}
This result is stated without proof in \cite{miles:pnas1}. A proof of
the corresponding fact for the intensity measure can be found in
\cite{santalo}, and another in \cite{kendall-et-al}, ch.~8. A short,
elementary proof is given in Appendix~\ref{sec:appendix}. The
following corollary, which is stated without proof as Theorem~2 in
\cite{miles:pnas1}, follows easily from isometry-invariance. 
\end{remark}

\begin{corollary}\label{corollary:line-intensity}
Let $\mathcal{L}$ be a Poisson line process of intensity $\kappa
>0$. For any fixed line  $\ell$ in $\mathbb{R}^{2}$, the point
process of intersections of $\ell$ with lines in $\mathcal{L}$ is a
Poisson point process of intensity $2\kappa /\pi$.
\end{corollary}

\begin{lemma}\label{lemma:local-sparsity}
  Let $\mathcal{L}$ be a Poisson line process of intensity
  $\kappa >0$, and for each point $x\in\mathbb{R}^{2}$ and each real
  $r>0$ let $N(B(x;r))$ be the number of lines in $\mathcal{L}$ that
  intersect the ball $B(x;r)$ of radius $r$ centered at $x$. Then the
  random variable $N(B(x;r))$ has the Poisson distribution with mean
  $2\kappa r$. Consequently, with probability one, for any compact
  set $K\subseteq \zz{R}^{2}$ the set of lines $L_{n}$ in
  $\mathcal{L}$ that intersect $K$ is finite.
\end{lemma}

\begin{proof}
Without loss of generality, take $K=B(0;R)$ to be the closed ball of radius
$R$ centered at the origin. Then the line $L_{n}$ intersects $K$ if
and only if $|R_{n}|\leq R$. Since a Poisson point process on $\zz{R}$
of constant intensity has at most finitely many points in any finite
interval, the result follows.
\end{proof}

The next result {characterizes} the
Poisson line process (see also
Proposition~\ref{proposition:sufficientCondition} below). Fix a
bounded, {convex} region $D\subset \mathbb{R}^{2}$ with $C^{\infty}$
boundary $\Gamma =\partial D$, and let $A,B$ be
non-intersecting closed arcs on $\Gamma$. For any line process
$\mathcal{L}$, let
\begin{equation}\label{eq:crossingCounts}
	 N_{\{A,B \}}= \# \,\text{lines that cross both}\; A \;\text{and}\;B.
\end{equation}
For any angle $\theta \in [-\pi /2,\pi /2]$, the set of lines that
intersect both $A$ and $B$ and meet the $x-$axis at angle
$\theta +\pi /2$ constitute an infinite strip that intersects the line
$\{re^{i\theta} \}_{r \in \zz{R}}$ in an interval; see
Figure~\ref{fig:Strip at angle theta} below.  Let
$\psi (\theta)=\psi_{A,B} (\theta)$ be the length of this interval,
and define
\begin{equation}\label{eq:beta}
	\beta_{A,B}=\frac{1}{\pi} \int_{-\pi /2}^{\pi /2}\psi
	(\theta)\,d\theta.
\end{equation}

\begin{figure}[h!]
 \centering
 \includegraphics[scale=1]{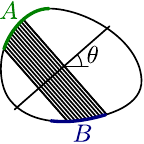}
 \caption{Lines that cross $A$ and $B$ at angle $\theta$.}
 \label{fig:Strip at angle theta}
\end{figure}

\begin{proposition}\label{proposition:characterization}
  A line process  $\mathcal{L}$  in $D$ is  a Poisson line process of rate
  $\kappa >0$ if and only if
  \begin{enumerate}
  \item [(i)]  for any two non-intersecting arcs $A,B \subset \Gamma$,
    the random variable $N_{\xset {A,B}}$ has the Poisson distribution
    with mean $\kappa \beta_{A,B}$, and
  \item [(ii)] for any finite collection  $\{A_{i},B_{i} \}_{i\leq
      m}$  of pairwise disjoint boundary arcs, the random
    variables $N_{\{A_{i},B_{i}\}}$ are mutually independent. 
  \end{enumerate}
\end{proposition}

See Appendix~\ref{sec:appendix} for the proof of the forward
implication, along with that of the following corollary. The converse
implication in Proposition~\ref{proposition:characterization}
will follow from Proposition~\ref{proposition:sufficientCondition} in
section~\ref{ssec:poissonApprox} below.

\begin{corollary}\label{corollary:ppIntensity}
Let $D\subset \zz{R}^{2}$ be a compact, convex region, and let $\mathcal{L}$ be a
Poisson line process with intensity $\kappa$.  The number $V(D)$ of
intersection points (vertices) of $\mathcal{L}$ in
$D$ has expectation
\[
	EV (D )=\kappa^{2}  |D|/\pi 
\]
where $|D|$ is the Lebesgue measure of $D$.
\end{corollary}

\subsection{Ergodic theorem for Poisson line
processes}\label{ssec:erg-theorem-plp}

The {configuration space} $\mathcal{C}$ in which a Poisson line
process takes values is the set of all countable, locally finite
collections of lines in $\zz{R}^{2}$. This space has a natural metric
topology, specifically, the weak topology generated by the Hausdorff
topologies on the restrictions to balls in $\zz{R}^{2}$. Moreover,
$\mathcal{C}$ admits an action (by translations) of
$\zz{R}^{2}$. Denote by $\nu_{\kappa}$ the distribution of the
Poisson line process with intensity $\kappa$. By
Lemma~\ref{lemma:translation-invariance}, the measure $\nu_{\kappa}$
is translation-invariant.

\begin{proposition}\label{proposition:ergodicity-plp}
The probability measure $\nu_{\kappa}$ is mixing (and therefore
ergodic) with respect to the translational action of $\zz{R}^{2}$ on
$\mathcal{C}$. 
\end{proposition}

\begin{remark}\label{remark:kendall}
Ergodicity of the measure $\nu_{\kappa}$ is asserted in Miles'
papers~\cite{miles:pnas1}, \cite{miles:pnas2}, and proved in his
unpublished Ph.~D. dissertation. We have been unable to locate a
proof in the  published literature, so we have provided one in the Appendix.
\end{remark}

\begin{corollary}\label{corollary:poly-freqs}
Let $\Phi _{n,k}$ be the fraction of $k-$gons, $F_{n}$ (for ``faces'') the
total number of polygons, and $V_{n}$ (for ``vertices'') the number of intersection
points in the tessellation of the
square $[-n,n]^{2}$ induced by a Poisson line process $\mathcal{L}$ of intensity
$\kappa$. There exist constants $\phi _{k}>0$ such that with 
probability $1$,
\begin{align}\label{eq:poly-freqs}
		&\lim_{n \rightarrow \infty} F_{n}/ (2n)^{2}=\kappa^{2} /\pi  ,\\
\label{eq:vertex-freqs}
 	&\lim_{n \rightarrow \infty} V_{n}/ (2n)^{2}=\kappa^{2}/\pi , \quad \text{and}\\
\label{eq:kgonFracs}
 	&\lim_{n \rightarrow \infty} \Phi _{n,k}=\phi _{k} .
\end{align}
\end{corollary}

Integral formulas  for the quantities $\phi_k$ are given
in~\cite{Calka}.

The ergodic theorem  can also be used to prove that a variety of  other statistical
properties stabilize in large squares. Consider, for example,  the number $N_{n}
(A,B,C)$ of triangles contained in $[-n,n]^{2}$ whose side lengths
$\alpha ,\beta ,\gamma$ lie in the intervals $A,B,C$; then as $n
\rightarrow \infty$,
\[
	 N_{n} (A,B,C)/ (2n)^{2} \longrightarrow E \mathbf{1}_{G (A,B,C)} (\mathcal{L})
\]
where $G (A,B,C)$ is the event that the polygon containing the  origin
is a triangle with side lengths in $A,B,C$.

\subsection{Weak convergence to a  Poisson line
process}\label{ssec:poissonApprox}

For any unordered pair $\{A,B \}$ of non-overlapping boundary arcs of
the disk $B (0,\alpha)$, let $L_{\{A,B \}}$ be the set of lines in
$\zz{R}^{2}$ that intersect both $A$ and $B$. This set can be
identified with the set of point pairs $\{x,y \}$ where $x\in A$ and $y\in
B$. This allows us to view any random collection of unordered point pairs $\{x,y
\}$ as a \emph{line process} in $B (0,\alpha)$, even when the
collection consists of endpoints of arcs across $B (0,\alpha)$ that
are not line segments (in particular, when they are pullbacks of
geodesic arcs to the tangent space). For any line process
$\mathcal{L}$ in $B (0,\alpha)$ let $N_{\{A,B \}}$ be the cardinality
of $\mathcal{L}\cap L_{\{A,B \}}$
(cf. equation~\eqref{eq:crossingCounts}).

\begin{proposition}\label{proposition:sufficientCondition}
Let $\mathcal{L}_{n}$ be a sequence of line processes in $B
(0;\alpha)$, and let $\mu_{n}$ be the distribution of
$\mathcal{L}_{n}$ (i.e., the probability measure on $\mathcal{X}$
induced by $\mathcal{L}_{n}$). In order that $\mu_{n}\rightarrow \mu$
weakly, where $\mu$ is the law of a rate$-\kappa$ Poisson line
process, it suffices that the following condition holds. For any
finite collection $\{\{ A_{i},B_{i}\} \}_{i\leq m}$ of unordered pairs
of non-overlapping boundary arcs of $B(0;\alpha)$ such that the sets
$L_{\{A_{i},B_{i} \}}$ are pairwise disjoint, the joint distribution
of the counts $N_{\{ A_{i},B_{i}\}}$ under $\mu_{n}$ converges to the joint
distribution under $\mu$, that is, for any choice of nonnegative
integers $k_{i}$,
\begin{equation}\label{eq:suff-weak-convergence}
	\lim_{n \rightarrow \infty} \mu_{n}\{N_{\{A_{i},B_{i} \}}=k_{i} \;
	\forall \,i\leq m\} =\prod_{i=1}^{m}
	\frac{(\kappa  \beta_{A_{i},B_{i}})^{k_{i}}}{k_{i}!}
	e^{-\kappa  \beta_{A_{i},B_{i}}} .
\end{equation}
\end{proposition}

\begin{proof}
[Proof Sketch] Recall that the configuration space $\mathcal{X}$ is the
disjoint union of the  sets $\mathcal{X}_{k}$, where
$\mathcal{X}_{k}$ is the set of all finite sets $F=\{\{x_{i},y_{i} \}
\}_{1\leq i\leq k}$ consisting of $k$ unordered pairs of points on
$\partial B (0,\alpha)$. Since each set $\mathcal{X}_{k}$ is both open
and closed in $\mathcal{X}$, to prove weak convergence
$\mu_{n}\rightarrow \mu$ it suffices to establish the
convergence \eqref{eq:weak-conv-definition} for every continuous
function $f$ supported by just one of the sets $\mathcal{X}_{k}$.

For each $k$, the space $\mathcal{X}_{k}$ is a quotient of $(\partial
B (0,\alpha)^{2})^{k}$ with the usual topology, and so every
continuous function $f:\mathcal{X}_{k}\rightarrow \zz{R}$ can be
uniformly approximated by ``step functions'', that is, functions $g$
of configurations $F=\{\{x_{i},y_{i} \} \}_{1\leq i\leq k}$ that
depend only on the counts $N_{A_{i},B_{i}}$ 
for arcs $A_{i},B_{i}$ in some partition of $\partial B (0,\alpha)$.
If \eqref{eq:suff-weak-convergence} holds, then it follows by
linearity of expectations that for any such step function $g$,
\[
	\lim_{n \rightarrow \infty} \int g\,d\mu_{n}=\int g\,d\mu ,
\]
and hence  \eqref{eq:weak-conv-definition}  follows.
\end{proof}

\subsection{The ``law of small numbers''}\label{ssec:lsn}

A elementary theorem of discrete probability theory states that for
large $n$, the Binomial$- (n,\kappa /n)$ distribution is closely
approximated by the Poisson distribution with mean
$\kappa$. Following is a generalization that we will find useful.

\begin{proposition}\label{proposition:lsn}
Let $X_{1},X_{2},\dotsc ,X_{n}$ be independent Bernoulli random
variables with success parameters $EX_{i}=p_{i}$. Let $\alpha
=\max_{i}p_{i}$ and $\kappa =\sum_{i}p_{i}$. Then there is a constant
$C<\infty$ not depending on $p_{1},p_{2},\dotsc ,p_{n}$ such that 
\[
	\sum_{k=0}^{\infty} \left\lvert P\left\{\sum_{i}X_{i}=k
	\right\}-\frac{\kappa^{k}}{k!}e^{-\kappa}\right\rvert \leq C\alpha . 
\]
\end{proposition}
Note that for all $k > n$, the probability $P\left\{\sum_{i}X_{i}=k \right\}$ is zero, but the elements of the sum are not. 

See \cite{lecam} for a proof. The important feature of the proposition
for us is not the explicit bound, but the fact that the closeness of
the approximation depends only on $\max p_{i}$. 

A similar result holds for multinomial variables.

\begin{proposition}\label{proposition:mult-lsn}
Let $X_{1},X_{2},\dotsc ,X_{n}$ be independent random variables each
taking values in the finite set $\{0,1,2,\dotsc ,K \}=\{0\} \cup [K]$, and for each
pair $i,j$ set $p_{i,j}=P\{X_{i}=j \}$. Let $\alpha =\max_{j\geq 1}
\max_{i}p_{i,j}$ and $\kappa_{j}=\sum_{i}p_{i,j}$, and for each $j$
define 
\[
	 T_{j}=\sum_{i=1}^{n}\mathbf{1}\{X_{i}=j \}.
\]
 Then there is a
function $C_{K} (\alpha)$ satisfying $\lim_{\alpha \downarrow 0}C
(\alpha)=0$ such that 
\[
	\sum_{m_{1}=0}^{\infty}\sum_{m_{2}=0}^{\infty}\dotsb
	\sum_{m_{K}=0}^{\infty} 
	\left\lvert P\{T_{j}=m_{j} \;\forall \, j\in [K]\}
	-\prod_{j=1}^{K}\kappa_{j}^{m_{j}}e^{-\kappa_{j}}/m_{j}! \right\rvert \leq C (\alpha).
\]
\end{proposition}

\section{Preliminaries: Symbolic Dynamics}\label{sec:symbolicDynamics}

\subsection{Shifts and suspension flows}\label{ssec:suspFlows}

The geodesic flow on the unit tangent bundle $T^1S$ of a compact,
negatively curved surface $S$ has a concrete representation as a
suspension flow over a shift of finite type. In describing this
representation, we shall follow (for the most part) the terminology
and notation of \cite{bowen:book}, \cite{series:symbDynFlows}, and
\cite{lalley:duke}. Let $\mathcal{A}$ be a finite alphabet and
$\mathcal{F}$ a finite set of finite words on the alphabet
$\mathcal{A}$, and define $\Sigma =\Sigma_{\mathcal{F}}$ to be the set
of doubly infinite sequences $\omega = (\omega_{n})_{n\in \zz{Z}}$
such that no element of $\mathcal{F}$ occurs as a subword of
$\omega$. The sequence space $\Sigma $ is given the metric
$d (\omega,y)=\exp \{-n (\omega,y) \}$ where $n (\omega,y)$ is the
minimum nonnegative integer $n$ such that $\omega_{j}\not =y_{j}$ for
$j= n$ or $j=-n$.  For each nonnegative integer $m$ and each $\omega
\in \Sigma$, define the \emph{cylinder set} $\Sigma_{m}(\omega) $ to
be the set of all $\omega'\in\Sigma$ that agree with $\omega$ in all
coordinates $j$ such that $|j| \leq m$; equivalently,
\begin{equation}
  \label{eq:cylinder}
  \Sigma_{m}(\omega)= \xset {\omega' \in \Sigma \,:\,
    d(\omega,\omega')< e^{-m}}.
\end{equation}
The forward shift $\sigma :\Sigma \rightarrow \Sigma$ is
known as a (two-sided) \emph{shift of finite type}.\footnote{Bowen
  \cite{bowen:book} requires that the elements of the set
  $\mathcal{F}$ all be of length 2. However, any shift of finite type
  can be ``recoded'' to give a shift of finite type obeying Bowen's
  convention, by replacing the original alphabet $\mathcal{A}$ by
  $\mathcal{A}^{m}$, where $m$ is the length of the longest word in
  $\mathcal{F}$, and then replacing each sequence $\omega $ by the
  sequence $\bar{\omega}$ whose entries are the successive length-$m$
  subwords of $\omega$. In Series' \cite{series:symbDynFlows} symbolic
  dynamics for the geodesic flow on a closed hyperbolic surface, the
  alphabet $\mathcal{A}$ is the set of natural generators for the
  fundamental group $\pi_{1} (S)$ of the surface $S$, and the
  forbidden subwords $\mathcal{F}$ are gotten from the relators of
  $\pi_{1} (S)$.}

For any continuous function $F:\Sigma \rightarrow (0,\infty)$ on
$\Sigma$, define the \emph{suspension space} $\Sigma_{F}$
by
\[
	\Sigma_{F} :=\{(\omega,t)\,:\, \omega\in \Sigma \;\; \text{and} \; 0\leq
	t\leq F (\omega)\} ,
\]
with points $(\omega,F (\omega))$ and $(\sigma \omega,0)$
identified. The metric $d$ on the sequence space $\Sigma$ induces a
metric $d_{\textsc{Taxi}}$ on $\Sigma_{F}$, the ``taxicab''
metric. (Roughly, the distance between any two points $(\omega, t)$
and $(\omega',t')$ in $\Sigma_{F}$ is the length of the shortest
``path'' between them consisting of alternating ``horizontal'' and
``vertical'' segments.  See \cite{bowen-walters} for the formal definition.) The
\emph{suspension flow} with height function $F$ is the flow $\phi_{t}$
on $\Sigma_{F}$ whose orbits proceed up vertical fibers
\begin{displaymath}
  \mathcal{F}_{\omega}:=\{(\omega,s)\,:\,0\leq s\leq F (\omega) \}
\end{displaymath}
at speed $1$, and upon reaching the ceiling at $(\omega,F (\omega))$
jump instantaneously to $(\sigma \omega,0)$. If the height function
$F: \Sigma \rightarrow \mathbb{R}$ is H\"{o}lder continuous with
respect to the metric $d$, then the suspension flow $\phi_{t}$ is
H\"{o}lder continuous with respect to the metric $d_{\textsc{Taxi}}$:
in particular, there exists $\alpha>0$ such that
\begin{equation}\label{eq:suspFlowHolder}
  d_{\textsc{Taxi}}(\phi_{t}(\omega,0), \phi_{t}(\omega',0))\leq
  e^{\alpha |t|}d(\omega,\omega') \quad \textrm {for all} \quad
  \omega,\omega'\in\Sigma \; \textrm {and} \; t \in \mathbb{R}.
\end{equation}

There is a bijective correspondence between invariant probability
measures $\mu^{*}$ for the flow $\phi_{t}$ and shift-invariant measures $\mu$ on
$\Sigma$. This correspondence can be specified as follows: for any
continuous function $g:\Sigma_{F} \rightarrow \zz{R}$, 
\begin{equation}\label{eq:measure-transfer}
	\int g\, d\mu^{*}=\int_{\Sigma}\int_{0}^{F (\omega)} g (\omega,s) \,ds
	\,d\mu (\omega) /\int_{\Sigma}F\,d\mu .
\end{equation}
If $\mu$ is ergodic  for the shift $(\Sigma ,\sigma)$ then $\mu^{*}$ is
ergodic for the flow $(\Sigma_{F},\phi_{t})$; and if $\mu$ is mixing
for the shift then  $\mu^{*}$ is mixing for the flow {provided} that the
height function $F$ is not cohomologous to a function $F'$ that takes
values in $b\zz{Z}$ for some $b>0$. (Two functions $F,F'$ are
\emph{cohomologous} if their difference is a \emph{coboundary}
$G-G\circ\sigma$.) By Birkhoff's
theorem, for any
ergodic  probability measure $\mu$,
\[
	\lim_{n \rightarrow \infty}n^{-1}\sum_{j=0}^{n-1}
	F\circ\sigma^{j} =\int_{\Sigma}F\,d\mu \quad \text{almost
	surely};  
\]
thus, under $\mu^{*}$, almost every orbit makes roughly $T/\int F
\,d\mu$ visits to the base $\Sigma \times \{0 \}$ by time $T$, when
$T$ is large.

\subsection{Symbolic dynamics for the geodesic flow}\label{ssec:symbDynGF}

The following proposition is a special case of the main result of
\cite{ratner:markov} (see also \cite{bowen:symbDyn}), as the geodesic
flow on a compact, negatively curved surface is an Anosov flow.

\begin{proposition}\label{proposition:boundaryCorrespondence}
For any compact, negatively curved surface
$(S, \varrho)$ with $C^{\infty}$ Riemannian metric $\varrho$, there
exist a topologically mixing shift $(\Sigma ,\sigma)$ of 
finite type, a suspension flow
$(\Sigma_{F},\phi_{t})$ over the shift with H\"{o}lder continuous
height function $F$, and a surjective,
H\"{o}lder-continuous mapping $\pi:\Sigma_{F} \rightarrow T^1S$ such that
$\pi$ is a semi-conjugacy with the geodesic flow $\gamma_{t}$ on
$T^1S$, i.e.,
\begin{equation}
  \label{eq:semi-conjugacy}
	\pi \circ \phi_{t}=\gamma_{t} \circ \pi \quad \text{for all}\;\;
        t\in \zz{R}.
\end{equation}
\end{proposition}

In the special case where $\varrho$ is a hyperbolic (constant
curvature) Riemannian metric, a much more explicit symbolic dynamics
was constructed by  Series:
see \cite{series:symbDynFlows}, especially Th. 3.1, and also
\cite{bowen-series}. In this symbolic dynamics, the sequence space
$\Sigma$ is mapped to a subset of  $\partial \mathbb{D}\times \partial
\mathbb{D}$, where $\partial
\mathbb{D} $ is the ideal    boundary of the Poincar\'{e} disk, in
such a way that every vertical fiber
$\mathcal{F}_{\omega}$  of the suspension flow is mapped to a segment
of the hyperbolic geodesic in $\mathbb{D}$ whose endpoints are gotten
from the boundary correspondence. Series' symbolic dynamics can be
extended to the variable curvature case using  the \emph{Conformal
  Equivalence Theorem} (\cite{schoen-yau}, 
Theorem V.1.3) and the structural stability theorem for Anosov
flows. This more explicit symbolic dynamics will not be needed in the
analysis below. However, we will need the following fact (see
\cite{paulin-pollicott-schapira}, ch.~7).

\begin{proposition}
  \label{proposition:LiouvilleGibbs}
  Under the hypotheses of
  Proposition~\ref{proposition:boundaryCorrespondence}, the pullback
  $\lambda^{*}:=\mu_{L}\circ \pi^{-1}$ of the normalized
  Liouville measure $\mu_{L}$ on $T^{1}S$ is Gibbs, that is, it corresponds to a
  Gibbs state $\lambda$ for the shift via the identity  \eqref{eq:measure-transfer}.
\end{proposition}

\subsection{Regenerative representation of Gibbs
states}\label{ssec:regeneration} 

Gibbs states with H\"{o}lder continuous potentials enjoy strong
exponential mixing properties (e.g., the ``exponential cluster
property'' 1.26 in \cite{bowen:book}, ch.~1). We shall make use of an
even stronger property, the \emph{regenerative representation} of a
Gibbs state established in \cite{lalley:regeneration} (cf. also
\cite{comets-et-al}). This representation is
most usefully described in terms of the stationary process governed by
the Gibbs state. Let $\mu$ be a Gibbs state with H\"{o}lder continuous
potential function $f:\Sigma  \rightarrow \mathbb{R}$, where $\sigma
:\Sigma  \rightarrow \Sigma$ is a topologically mixing shift of finite
type, and let $X_{n}:\Sigma  \rightarrow \mathcal{A}$ be the
coordinate projections on $\Sigma$, for $n\in \zz{Z}$. The sequence
$(X_{n})_{n\in \zz{Z}}$, viewed as a stochastic process on the
probability space $(\Sigma ,\mu)$, is a stationary process that we
will henceforth call a \emph{Gibbs process}. 

The regenerative representation relates  the class of Gibbs processes to
another class of stationary processes, called \emph{list processes}
(the term used by \cite{lalley:regeneration}). A list process is a stationary,
positive-recurrent Markov chain $(Z_{n})_{n\in \zz{Z}}$ with state
space $\cup_{k\geq 1}\mathcal{A}^{k}$  and stationary distribution
$\nu$ that obeys the following transition rules: first,
\begin{equation}\label{eq:list-1}
	P (Z_{n+1}= (\omega_{1},\omega_{2},\dotsc ,\omega_{m})\,|\,
	Z_{n}= (\omega '_{1},\omega '_{2},\dotsc , \omega '_{k}))=0
\end{equation}
unless either $m=1$ or $m=k+1$ and $\omega_{i}=\omega '_{i}$ for each
$1\leq i\leq k$; and second, for every letter $\omega_{1}$ and every
word $\omega '_{1}\omega '_{2}\dotsb \omega '_{m}$,
\begin{equation}\label{eq:list-2}
	P (Z_{n+1}=\omega_{1}\,|\, Z_{n+1}\in \mathcal{A}^{1} \;
	\text{and}\; Z_{n}= (\omega '_{1},\omega '_{2},\dotsb ,\omega
	'_{m})) =\nu ((\omega_{1}))/\nu (\mathcal{A}^{1}).
\end{equation}
Thus, the process $(Z_{n})_{n\in \zz{Z}}$ evolves by either adding one
letter to the end of the list or erasing the entire list and
beginning from scratch.  Furthermore, by \eqref{eq:list-2}, at any
time when the list is erased, the new 1-letter word chosen to begin
the next list  is independent of the past history of the entire process.

For any list process define the
\emph{regeneration times} $0=\tau_{0}<\tau_{1}<\tau_{2}<\dotsb$ by
\begin{align*}
	\tau_{1}&=\min \{n\geq 1\,:\, Z_{n}\in \mathcal{A}^{1}\};\\
	\tau_{m+1}&=\min \{n\geq 1+\tau_{m}\,:\, Z_{n}\in \mathcal{A}^{1}\}.
\end{align*}
By condition \eqref{eq:list-2}, the random variables
$\tau_{m+1}-\tau_{m}$ are independent, and    for
$m\geq 1$ are identically distributed, as are the \emph{excursions} 
\[
	(Z_{\tau_{m}+1},Z_{\tau_{m}+2},\dotsc ,Z_{\tau_{m+1}}).
\]
Denote by $\pi :\cup_{k\geq 1}\mathcal{A}^{k} \rightarrow \mathcal{A}$
the projection onto the \emph{last} letter.
 
\begin{proposition}\label{proposition:regeneration}
If $(X_{n})_{n\in \zz{Z}}$ is a Gibbs process then there is a list process
$(Z_{n})_{n\in \zz{Z}}$ such that the projected process $(\pi
(Z_{n}))_{n\in \zz{Z}}$ has the same joint distribution as the Gibbs
process $(X_{n})_{n\in \zz{Z}}$ . Thus, the random sequence obtained
by concatenating the successive excursions $W_{m}:=Z_{\tau_{m}}$,
i.e.,
\[
	W_{1}\cdot W_{2}\cdot W_{3}\cdot\dotsb ,
\]
has the same distribution as the sequence $\{X_{n} \}_{n \geq 0}$.
Moreover, the list process can be chosen
in such a way that the excursion lengths $\tau_{m+1}-\tau_{m}$ satisfy 
\begin{equation}\label{eq:exp-regeneration}
	P (\tau_{m+1}-\tau_{m}\geq n)\leq C\alpha^{n}
\end{equation}
for some $0<\alpha <1$ and $C<\infty$ not depending on either $m$ or
$n$. 
\end{proposition}

See \cite{lalley:regeneration}, Th.~1, or \cite{comets-et-al},
Th.~4.1. (The  article \cite{lalley:regeneration} uses the (older) term \emph{chain with
  complete connections} for a Gibbs process, and a different (but
equivalent) definition than that given in \cite{bowen:book}. Moreover,
\cite{lalley:regeneration} considers only the case where the
underlying shift is the full shift on the symbol set $\mathcal{A}$,
although the proof extends routinely to the general case. See
\cite{li-naud-pan} for details.)

\section{Theorem~\ref{theorem:random-local} Proof:
  Strategy}\label{sec:heuristics}

We shall use the symbolic dynamics outlined in
section~\ref{ssec:symbDynGF} to translate the weak convergence problem
to a problem involving the Gibbs state $\lambda$ corresponding to the
pullback $\lambda^{*}$ of Liouville measure to the suspension space
$\Sigma_{F}$. Recall
(cf. Proposition~\ref{proposition:boundaryCorrespondence}) that the
projection $\pi:\Sigma_{F}\rightarrow T^{1}S$ provides a
semi-conjugacy \eqref{eq:semi-conjugacy} between the suspension flow
$\phi_{t}$ and the geodesic flow $\gamma_{t}$; thus, each segment
of the suspension flow projects (via the mapping
$p\circ\pi$, where $p:T^{1}S\rightarrow S$ is the natural projection)
to a geodesic segment of the same length, and in
particular, each fiber $\mathcal{F}_{\omega}$ of the suspension space
$\Sigma_{F}$ projects to a geodesic segment of length $F(\omega)$.  By
Birkhoff's ergodic theorem, for $\lambda-$almost every $\omega \in \Sigma$
the length of the orbit segment
\begin{displaymath}
  \mathcal{F}_{\omega}\cup \mathcal{F}_{\sigma \omega}\cup \cdots \cup
  \mathcal{F}_{\sigma^{n-1}\omega}
\end{displaymath}
divided by $n$ converges as $n \rightarrow\infty$ to $E_{\lambda}F$.
We will show (cf. Proposition~\ref{proposition:measure-asymptotics}
below) that for a random geodesic the expected number of visits to the
region $D(x;\alpha T^{-1})$ by time $T$ is of order $1$, and
(cf. Proposition~\ref{proposition:no-quick-returns}) that  the
expected number of visits in a time interval of length $\varepsilon T$
can be made arbitrarily small by taking $\varepsilon$
small. Therefore, the intersection  of a length-$T$ random
geodesic  with  $D(x,\alpha T^{-1})$ is, with probability
approaching one, identical to the intersection with the geodesic segment
\begin{equation}
  \label{eq:n}
  p\circ\pi(\cup _{i=0}^{n-1}\mathcal{F}_{\sigma^{i}\omega}) \quad
         \textrm {where} \quad n:=n (T)= [T/E_{\lambda}F].
 \end{equation}
Henceforth, we will use the abbreviation $n=n(T)$.

Denote by $I_{n} (x,\omega)$ the intersection of the geodesic segment
$p\circ\pi(\cup _{i=0}^{n-1}\mathcal{F}_{\sigma^{i}\omega}) $ with the
neighborhood $D(x;\alpha T^{-1})$, and by
$J_{n}(x, \omega)$ the pullback to a finite collection of chords of
the ball $B(0, \alpha)$ in the tangent space $T_{x}S$ (cf. the
discussion preceding the statement of
Theorem~\ref{theorem:random-local}).  Our goal is to prove that, for
any fixed $x\in \mathcal{S}$, the sequence of line processes $J_{n}$
converges in law to a Poisson line process on $B (0,\alpha)$. For this
we will use the criterion of
Proposition~\ref{proposition:sufficientCondition}.

For any pair $A,B$ of non-overlapping boundary arcs of
$\partial B (0,\alpha)$, define $L_{A,B}$ to be the set of
\emph{oriented} line segments from $A$ to $B$, and let
$N_{A,B} (\omega)$ be the number of oriented chords in
$J_{n} (x,\omega)$ from boundary arc $A$ to boundary arc $B$ in
$B(0,\alpha)$, equivalently, the number of oriented geodesic segments
in the collection $I_{n} (x;\omega )$ that cross the target
neighborhood $D(x;\alpha T^{-1})$ from (the image of) arc
$A$ to (the image of) arc $B$. (Recall that $B (0,\alpha)$ is
identified with the neighborhood $D(x,\alpha T^{-1})$ by the
scaled exponential mapping. Henceforth, for any pair of arcs $A,B$ in
$\partial B(0 , \alpha)$ we shall denote by $A^{T},B^{T}$ the
corresponding boundary arcs of $D(x,\alpha T^{-1})$.) The
counts $N_{A,B}$ depend on $n=[T/E_{\lambda}F]$ and $\omega$, but to
reduce notational clutter we shall suppress this dependence.  Observe
that the number of \emph{undirected} crossings $N_{\{A,B \}}$
(cf. equation~\eqref{eq:crossingCounts}) is given by
\[
	N_{\{A,B \}}=N_{A,B}+N_{B,A},
\]
and consequently $EN_{\{A,B \}}=EN_{A,B}+EN_{B,A}$. Since 
the sum of independent Poisson random variables is Poisson, 
to prove that in the $n \rightarrow \infty$ limit the random variable
$N_{\{A,B\}}$ becomes Poisson, it suffices to show that the
\emph{directed} crossing counts $N_{A,B}$ become Poisson. Thus,
 our objective now is to prove the following assertion, which, by
Proposition~\ref{proposition:sufficientCondition}, will
imply Theorem~\ref{theorem:random-local}.

\begin{proposition}\label{proposition:reformulation}
For any finite collection $\{(A_{i},B_{i}) \}_{i\leq r}$ of pairs of
non-overlapping closed boundary arcs of $B (0,\alpha)$ such that the
sets $L_{A_{i},B_{i}}$ are pairwise disjoint, and for any choice of
nonnegative integers $k_{i}$,
\begin{equation}\label{eq:reformulation}
	\lim_{n \rightarrow \infty} \lambda \left\{\omega \,:\,
	N_{A_{i},B_{i}} (\omega )=k_{i} \; \forall \,i \right\}
	 =\prod_{i=1}^{r} \frac{( \kappa 
	\beta_{A_{i},B_{i}}/2)^{k_{i}}}{k_{i}!}  e^{- \kappa \beta_{A_{i},B_{i}}/2} 
\end{equation}
where $\beta_{A,B}$ is  defined by
equation~\eqref{eq:beta} (with $D=B(0,a)$) and $\kappa=1/\textrm {area}(S)$.
\end{proposition}

Note that for fixed boundary arcs $A,B$ the constants
$\beta_{A,B}=\beta_{A,B}(\alpha)$ are proportional to  $\alpha$, because
the function $\psi =\psi_{A,B}$ in \eqref{eq:beta} is proportional to
$\alpha$. See Figure~\ref{fig:Strip at angle theta}.

The proof of Proposition~\ref{proposition:reformulation} will be
accomplished in four stages, as follows.

First, we will prove in section~\ref{sec:no-prefixes} that for any
positive function $f(T)$ satisfying
\begin{displaymath}
  \lim_{T \rightarrow  \infty}f(T)/T=0,
\end{displaymath}
\begin{enumerate}
\item [(a)] the probability that a random geodesic ray enters the
  neighborhood $D(x, \alpha T^{-1})$ before time $f(T)$
  converges to $0$ as $T \rightarrow\infty$ (meaning there are no quick entries), and
\item [(b)] the probability that a random geodesic ray enters
  $D(x, \alpha T^{-1})$ before time $T$ and then re-enters
  (after having exited) within time $f(T)$ also converges to $0$ (meaning there are no quick re-entries).
\end{enumerate}
This will justify the replacement of the random length-$T$ geodesic
segment in the statement of Theorem~\ref{theorem:random-local}
 by the geodesic segment \eqref{eq:n} above, and will
also ultimately be used to partition this segment into nearly
independent blocks.

Second, define $\Sigma (A,B;T)$ to be the set of all sequences
$\omega \in \Sigma$ such that the geodesic segment
$p\circ\pi(\mathcal{F}_{\omega})$ intersects
$D(x, \alpha T^{-1})$ in a geodesic segment with terminal
endpoint in the boundary arc $B^{T}$, and either coincides with or
extends to a geodesic crossing from boundary arc $A^{T}$ to boundary
arc $B^{T}$. (Note that if the image $p\circ\pi (\omega,0)$ of the
base point $(\omega,0)$ lies in the interior of
$D(x, \alpha T^{-1})$ then the intersection will only be a
partial crossing.)  By assertions (a) and (b) above, the event
$\xset {\omega \,:\, N_{A,B}(\omega)=k}$ coincides (up to a set of
measure $\rightarrow 0$ as $T \rightarrow \infty$) with the set of
sequences $\omega \in \Sigma$ such that
\begin{displaymath}
  \sum_{i=0}^{n-1}\mathbf{1} _{\Sigma(A,B;T)}(\sigma^{i}\omega)=k,
\end{displaymath}
that is, sequences whose forward $\sigma-$orbits
$(\sigma^{i}\omega)_{i \geq 0}$ make exactly $k$ visits to the set
$\Sigma(A,B;T)$ for $i \leq n-1$.  We will prove, in
section~\ref{sec:asymptotics}), that the set $\Sigma(A,B;T)$ has
$\lambda-$measure satisfying
\begin{equation}\label{eq:measure-asymptotics}
	\lim_{T \rightarrow \infty} T\lambda (\Sigma (A,B;T))=
	\frac{1}{2} \kappa \beta_{A,B}E_{\lambda}F.
\end{equation}

Third, in section~\ref{sec:decomposition}, we will show that the set
$\Sigma (A,B;T)$ can be represented approximately as a finite union of
cylinder sets $\Sigma_{m}(\omega)$. This will be done in such a way
that the lengths of the words defining the cylinder sets satisfy
$m= (\log n)^{2} =C' (\log T)^{2}$. It will then follow that the set
$\xset {\omega \,:\, N_{A,B}(\omega)=k}$ is (approximately) the set of
all sequences $\omega \in \Sigma$ whose first $n$ letters contain
exactly $k$ occurrences of one of the length-$2m+1$ sub-words
\begin{equation}\label{eq:magic-subword}
	\omega_{-m}\omega_{-m+1} \cdots \omega_{m}
\end{equation}
that define the cylinder sets $\Sigma_{m}(\omega)$.

Finally, in section~\ref{sec:final-step}, we will use the results of
steps 1, 2, and 3 to show that the number $N_{A,B} $ of crossings
through arcs $A,B$ on $\partial D(x;\alpha T^{-1})$ equals
(with high probability) the number of length-$(\log T)^{2}$ blocks
that contain one of the magic subwords, and (using the regeneration
theorem of section \ref{ssec:regeneration}) that these occurrence
events are independent small-probability events. Furthermore, we will
show that for distinct pairs $(A_{i},B_{i})$ of boundary arcs the
counts $N_{A_{i},B_{i}}$ are (approximately) independent.  The desired
result \eqref{eq:reformulation} will then follow from the Poisson
convergence criterion of section \ref{ssec:poissonApprox}.

The strategy just outlined is easily adapted to
Theorem~\ref{theorem:random-local-pairs}. Fix distinct points $x,x'\in
S$. For any pair $A,B$ of non-overlapping boundary arcs of $\partial B
(0,\alpha)$, denote by $N_{A,B} (\omega)$ and $N'_{A,B} (\omega)$ the
numbers of geodesic arcs in the collections $I_{n} (x)$ and $I_{n}
(x')$, respectively, that cross the target regions $D(x;\alpha T^{-1})$
and $D(x';\alpha T^{-1})$  from arc $A$ to arc $B$. To prove
Theorem~\ref{theorem:random-local-pairs} it suffices to prove the
following.

\begin{proposition}\label{proposition:reformulation-pairs}
For any finite collections $\{(A_{i},B_{i}) \}_{1\leq i\leq r}$ and
$\{(A'_{i},B'_{i}) \}_{1\leq i\leq r'}$ and any choice of nonnegative
integers $k_{i},k'_{i}$,
\begin{multline}\label{eq:reformulation-pairs}
	\lim_{n \rightarrow \infty} \lambda \left\{\omega \,:\,
		N_{A_{i}B_{i}} (\omega )=k_{i} \;\text{and}\;
		N'_{A'_{i}B'_{i}} (\omega )=k'_{i} \; \forall
		\,i\right\}\\
		=\left(\prod_{i=1}^{r} \frac{1}{k_{i}!} ( \kappa \beta_{A_{i},B_{i}}/2)^{k_{i}}
 e^{- \kappa \beta_{A_{i},B_{i}}/2}  \right) \left(\prod_{i=1}^{r'} \frac{1}{k'_{i}!}( \kappa  
	\beta_{A'_{i},B'_{i}}/2)^{k'_{i}}
	e^{- \kappa  \beta_{A'_{i},B'_{i}}/2}  \right). 
\end{multline}
\end{proposition}

\noindent
\textbf{Note:} To avoid notational clutter, here and throughout
sections \ref{sec:no-prefixes}, \ref{sec:asymptotics}, and \ref{sec:decomposition}
 we will use the abbreviation $\kappa$ for
$\kappa_{S}=1/\textrm {area}(S)$.

\section{No Quick Entries or Re-entries of Small
  Disks}\label{sec:no-prefixes}

\begin{proposition}\label{proposition:no-quick-entries}
  Let $\gamma$ be a geodesic ray whose initial tangent vector is
  chosen at random according to normalized Liouville measure. For any
  positive function $f(T)$ satisfying
  $\lim_{T\rightarrow \infty}f(T)/T=0$, the probability that the
  $\gamma$ enters the region  $D(x,\alpha T^{-1})$ before time
  $f(T)$  is of order $O(f(T)/T)$.
\end{proposition}

\begin{proof}
  For any unit vector $v \in T^{1}S$, denote
  by $\tau (v)$ the smallest nonnegative time $t$ (possibly $+\infty$)
  at which the geodesic ray $\gamma_{t}(v)$ with initial tangent
  vector $v$ enters $D(x,\alpha T^{-1})$. Since any geodesic that
  enters $D(x, \alpha T^{-1})$ must spend at least $2\alpha T^{-1}$
  units of time in the surrounding ball $D(x, 2 \alpha T^{-1})$, we
  have, by the invariance of the Liouville measure,
  \begin{align*}
    \mu_{L} \xset {v\,:\, \tau (v)\leq f(T)}&\leq (2\alpha)^{-1}T
    \int_{T^{1}S} \int_{0}^{f(T)}\mathbf{1} \xset { \gamma_{t}(v)\in D(x;
      2\alpha T^{-1}) }\, dt \, d\mu_{L}(v) \\
    &= (2\alpha)^{-1}T \int_{0}^{f(T)}  \int_{T^{1}S}  \mathbf{1} \xset { \gamma_{t}(v)\in D(x;
      2\alpha T^{-1}) }\, d \mu_{L}(v) \,dt \\
    &=  (2\alpha)^{-1}T \int_{0}^{f(T)}  \textrm {area} (D(x;
      2\alpha T^{-1})) \, dt/ \textrm {area}(S)\\
    &\sim  (2\alpha)^{-1}T \times f(T)\times (4\pi
      \alpha^{2}T^{-2}) / \textrm {area}(S)=O(f(T)/T)\longrightarrow 0.
  \end{align*}
  \end{proof}

\begin{proposition}
  \label{proposition:no-quick-returns}
  Let $\gamma$ be a geodesic ray whose initial tangent vector is
  chosen at random according to normalized Liouville measure. If
  $f(T)=o(T)$ as $T \rightarrow \infty$ then the probability that
  $\gamma$ enters (or begins in) $D(x,\alpha T^{-1})$ before
  time $T$ and then re-enters within time $f(T)$ converges to $0$ as
  $T \rightarrow \infty$.
\end{proposition}

As in the proof of Proposition~\ref{proposition:no-quick-entries}, the
region $D(x, \alpha T^{-1})$ can be replaced by the ball
$D(x, \alpha T^{-1})$. The proof of
Proposition~\ref{proposition:no-quick-returns} will be based on the
following estimate.

\begin{lemma}
  \label{lemma:no-quick-returns}
  Fix $y \in D(x, 2\alpha T^{-1})$ and $0< \beta <1$, and define
  $R_{T}=R_{T}(y)$ to be the set of unit tangent vectors $v$ based at
  $y$ such that the geodesic ray with initial tangent vector $v$
  enters the ball $D(x, \alpha T^{-1})$ before time $f(T)$ after having
  first exited $D(x, 2 \alpha T^{-1})$. There is a
  constant $K=K(\alpha)< \infty$ not depending on $y$ such that for all
  $T \geq 1$ the Lebesgue measure of $R_{T}$ satisfies
  \begin{equation}
    \label{eq:no-quick-returns}
    \textrm {Lebesgue} (R_{T})\leq K  f(T)/T .
  \end{equation}
\end{lemma}

The proof is deferred until after the proof of
Proposition~\ref{proposition:no-quick-returns}.  Given the lemma,
Proposition~\ref{proposition:no-quick-returns} follows by an
argument similar to that used in the proof of
Proposition~\ref{proposition:no-quick-entries}.

\begin{proof}[Proof of Proposition \ref{proposition:no-quick-returns}]
  Denote by $B_{T}$ the set of all unit tangent vectors $v$ with base
  point in $D(x, \alpha T^{-1})$ such that the geodesic ray
  $\gamma_{t}(v)$ re-enters $D(x, \alpha T^{-1})$  before time
  $f(T)$ (after having first
  exited), and by $B^{*}_{T}$ the set of all
  unit tangent vectors $v$ with base point in the enlarged disk
  $D(x, 2\alpha T^{-1})$ such that the geodesic ray $\gamma_{t}(v)$
  enters $B_{T}$ before leaving $D(x, 2\alpha T^{-1})$.  Clearly,
  $B_{T}\subset B^{*}_{T}$, and if $v \in B^{*}_{T}$, then the
  geodesic ray must spend at least $\alpha T^{-1}$ units of time in
  the disk $D(x, 2 \alpha T^{-1})$ before exiting. Moreover,
  Lemma~\ref{lemma:no-quick-returns} implies that
  \begin{displaymath}
    \mu_{L}(B^{*}_{T})\leq K  T^{-1} f(T)\times \frac{\textrm
    {area}(D(x, 2 \alpha T^{-1}))}{\textrm {area}(S)} \leq K'
  \alpha^{2} T^{-3} f(T).
  \end{displaymath}

  For any $v\in T^{1}(S)$ let $\tau^{*}(v)$ be the first time $t$ that the
  geodesic ray $\gamma_{t}(v)$ enters the set $B_{T}$. Then by the
  invariance of Liouville measure,
  \begin{align*}
    \mu_{L}\xset {v \,:\, \tau^{*}(v)\leq T}  & \leq \alpha^{-1}T
                                                \int_{T^{1}S}
                                                \int_{0}^{T} \mathbf{1}
                                                \xset {\gamma_{t}\in
                                                B^{*}_{T}} \, dt \,
                                                d\mu_{L}(v)\\
    & =   \alpha^{-1}T\int_{0}^{T}  \int_{T^{1}S} \mathbf{1}
                                                \xset {\gamma_{t}\in
                                                B^{*}_{T}} \, d
      \mu_{L}(v) \,dt\\
                                              &=  \alpha^{-1}T \int_{0}^{T}\mu_{L}(B^{*}_{T}) \, dt \\
    &\leq K' \alpha^{2} T^{-1} f(T)\longrightarrow 0 \quad \textrm
      {as} \; T \rightarrow \infty.
  \end{align*}
\end{proof}

\begin{proof}
  [Proof of Lemma~\ref{lemma:no-quick-returns}] Let $\tilde{S}$ be the
  universal cover of $S$, viewed as the (open) unit disk $\mathbb{D}$
  endowed with Riemannian metric $\tilde{\varrho}$, the natural lift
  of the Riemannian metric $\varrho$ on $S$. The metric
  $\tilde{\varrho}$ is invariant by the fundamental group
  $\pi_{1}(S)$. Furthermore, the action of $\pi_{1}(S)$ on $\tilde{S}$
  is discrete, and  there is a \emph{fundamental polygon}
  $\mathcal{P}$ for this action, bounded by geodesic segments, such
  that $\tilde{S}$ is tiled by the isometric  images $g\mathcal{P}$,
  where $g$ ranges over $\pi_{1}(S)$. Since the surface $S$ is compact, the
  fundamental polygon $\mathcal{P}$ can be chosen so that it has finite
  diameter $\delta$. Fix pre-images $\tilde{x}, \tilde{y}\in \tilde{S}$ of
  the points $x,y\in S$ in such a way that $\tilde{x}\in \mathcal{P}$ and
  $\tilde{y}\in D(\tilde{x}, 2 \alpha T^{-1})$; then for all
  sufficiently large $T$ the pre-image of the disk
  $D(x, \alpha T^{-1})$ is the disjoint union of isometric disks
  $D(g \tilde{x}, \alpha T^{-1})$ where $g$ ranges over
  $\pi_{1}(S)$. Clearly, since these disks are non-overlapping, only
  finitely many can intersect the fundamental polygon.

  The set $R_{T}(y)$ lifts to a set
  $\tilde{R}_{T}(\tilde{y})$ of the same Lebesgue measure in the unit
  tangent space $T^{1}_{\tilde{y}}(\tilde{S})$; this lift
  $\tilde{R}_{T}(\tilde{y})$ contains all direction vectors
  $v\in T^{1}_{\tilde{y}}(\tilde{S})$ such that the geodesic ray
  $\tilde{\gamma}_{t}(v)$ in $\tilde{S}$ with initial tangent vector
  $v$ intersects one of the balls $D(g \tilde{x}, \alpha T^{-1})$ with
  $g \not = 1$ at distance $\leq f(T)$ from $\tilde{y}$. To
  estimate the size of $\tilde{R}_{T}(\tilde{y})$, we decompose it by
  grouping the target disks $D(g\tilde{x}, \alpha T^{-1})$ in
  concentric shells by distance
  from the point $\tilde{y}$: thus, in particular,
  \begin{displaymath}
    \tilde{R}_{T}(\tilde{y})\subset
    \bigcup_{m=1}^{[f(T)+1]/(3\delta)}\bigcup_{g \in \mathcal{A}_{m}}
    \Theta_{g,T} (\tilde{y}) 
  \end{displaymath}
  where $\Theta_{g,T}(\tilde{y})$ is the set of all unit tangent
  vectors $v \in T^{1}_{\tilde{y}}(\tilde{S})$ such that the geodesic
  ray $\tilde{\gamma}_{t}(v)$ intersects  the disk $D(g\tilde{x},
  \alpha T^{-1})$, and
  \begin{displaymath}
    \mathcal{A}_{m}:=\xset {g\in \pi_{1}(S) \,:\, 3(m-1)\delta \leq\textrm
      {distance}(g \tilde{x}, \tilde{y}) < 3m\delta}. 
  \end{displaymath}
  (Recall that $\delta$ is the diameter of the fundamental
  polygon. Consequently, every point on the circle $\Gamma_{3m
    \delta}$ of radius $3 m \delta$ centered at $\tilde{y}$ is within
  distance $5\delta/2$ of  $g\tilde{x}$ for some $g \in \mathcal{A}_{m}$.)
  To prove inequality \eqref{eq:no-quick-returns} it will suffice to
  prove that for some constant $K$ independent of $m,T$ and the choice
  of $\tilde{y}\in D(\tilde{x}, \alpha T^{-1})$,
  \begin{equation}\label{eq:shadow}
    \sum_{g \in \mathcal{A}_{m}}\textrm {Lebesgue} (\Theta_{g,T})\leq K  T^{-1}.
  \end{equation}

  For each unit tangent vector $v \in T^{1}_{\tilde{y}}\tilde{S}$ and
  each real $t>0$, define the \emph{expansion factor} $\eta_{t}(v)$
  for the geodesic flow at time $t$ in direction $v$ to be the amount
  by which the exponential map $\exp _{\tilde{y}}$ expands distances
  at the tangent vector $tv$ in the direction $w=w(v)\in T^{1}S$
  orthogonal to $v$, that is,  
  \begin{displaymath}
    \eta_{t}(v)= \xnorm{(d\exp_{\tilde{y}}(t v))w}
  \end{displaymath}
  where $w=w(y)\in  T^{1}_{\tilde{y}}\tilde{S}$  is  the unit vector orthogonal
  to $v$ (the choice of sign is irrelevant). Note that since our surface is compact, this quantity is bounded away from both $1$ and $\infty$ (since the curvature is bounded away from zero). Thus, if
  $\Gamma_{t}=\Gamma_{t}(\tilde{y})$ is the circle of radius $t$ centered
  at  $\tilde{y}$ in $\tilde{S}$, then
  \begin{equation}\label{eq:expansionFormulaC}
    \textrm {circumference}(\Gamma_{t})= \int_{T^{1}_{\tilde{y}}\tilde{S}}
    \eta_{t}(v) \, d \textrm {Lebesgue}(v),
  \end{equation}
  and more generally, for any interval
  $\Theta \subset T^{1}_{\tilde{y}}\tilde{S}$,
  \begin{equation}\label{eq:expansionFormulaG}
    \textrm {arc-length} (\exp_{\tilde{y}}(t \Theta))=\int_{\Theta}
    \eta_{t}(v) \, d \textrm {Lebesgue}(v). 
  \end{equation}

  \begin{claim}\label{claim:expansionFactor}
    There exists a constant $C<\infty$ not depending on 
    the choice of      $\tilde{y}\in \tilde{S}$ such that for any two
    unit vectors $v_{1},v_{2}\in T^{1}_{\tilde{y}}\tilde{S}$
    satisfying the condition $\textrm
    {distance}(\tilde{\gamma}_{t}(v_{1}),
    \tilde{\gamma}_{t}(v_{2}))\leq 3\delta$, any $t
    > 3 \delta$, and any $0\leq s\leq 3\delta$,
    \begin{equation}
      \label{eq:expansionBounds}
      C^{-1}\leq \frac{\eta_{t}(v_{1})}{\eta_{t-s}(v_{2})} \leq     C .
    \end{equation}
  \end{claim}

  Before proving the claim, we show how it implies inequality
  \eqref{eq:shadow}. For each deck transformation $g\in \mathcal{A}_{m}$, let
  $\Theta^{*}_{g,m}$ be the set of all direction vectors $v$ based at
  $\tilde{y}$ for which the geodesic ray $\tilde{\gamma}_{t}(v)$
  approaches the point $g \tilde{x}$ at least as closely as it
  approaches any other $g' \tilde{x}$, where $g' \in \mathcal{A}_{m}$. These
  sets overlap either in isolated points or not at all, and their
  union is the entire set $T^{1}_{\tilde{y}}\tilde{S}$. Thus, the sets
  $\Theta^{*}_{m}$, where $g \in \mathcal{A}_{m}$, form (up to a set of measure
  $0$) a partition of $T^{1}_{\tilde{y}}\tilde{S}$. Furthermore, since
  the action of $\pi_{1}(S)$ on the universal cover $\tilde{S}$ is
  discrete, there exists an integer $k \geq1$ such that none of the
  sets $\Theta^{*}_{m}$ intersects more than $k$ of the arcs
  $\Theta_{g,T}$.

  Claim~\ref{claim:expansionFactor}, together with the identity
  \eqref{eq:expansionFormulaG}, implies that for a
  suitable constant $1<C_{1}<\infty$ not depending on $T,\tilde{y},m$, or
  $g\in \mathcal{A}_{m}$,
  \begin{align*}
    \textrm {Lebesgue} (\Theta_{g,T})&\leq C_{1}  T^{-1}
                                       \eta_{3m \delta}(v_{g})^{-1} \quad
                                       \textrm {and}\\
    \textrm {Lebesgue} (\Theta^{*}_{g,m})&\geq C_{1}^{-1}\eta_{3m
                                           \delta}(v_{g})^{-1}. 
  \end{align*}
  where $v_{g}\in T^{1}_{\tilde{y}}\tilde{S}$ is the unique direction
  such that the geodesic ray $\tilde{\gamma}_{t}(v_{g})$ goes through
  the point $g\tilde{x}$. Since each arc $\Theta_{g,T}$ is contained
  in the union (over $g' \in \mathcal{A}_{m}$) of the sets $\Theta^{*}_{g',m}$,
  and since no $\Theta^{*}_{g',m}$ intersects more than $k$ of the
  arcs $\Theta_{g,T}$, it follows that
  \begin{displaymath}
    \sum_{g \in \mathcal{A}_{m}}\textrm {Lebesgue} (\Theta_{g,T})\leq kC_{1}^{2} T^{-1}
     \sum_{g \in \mathcal{A}_{m}}\textrm {Lebesgue} (\Theta^{*}_{g,m}).
   \end{displaymath}
   The desired result \eqref{eq:shadow} now follows, because the sets
   $\Theta^{*}_{g,m}$ partition (up to a set of measure $0$) the unit
   tangent space $T^{1}_{\tilde{y}}\tilde{S}$.

\begin{proof}[Proof of Claim~\ref{claim:expansionFactor}]
  The expansion factor
  $\eta_{t}(v)$ can be calculated by integrating the infinitesimal
  expansion rates along the geodesic:
  \begin{gather*}
    \eta_{t}(v)= \exp \xset {\int_{0}^{t} \zeta _{s}(v) \, ds }\quad
      \textrm {where} \\
     \zeta_{t}(v)= \log \frac{d}{ds} 
     \left(
       \xnorm{d\exp_{\tilde{\gamma}_{t}(v)} (s \tilde{\gamma}_{t}'(v))w_{t}}
     \right)_{s=0}
   \end{gather*}
   and $w_{t}\in T^{1}_{\tilde{\gamma}_{t}(v)}\tilde{S}$ is the unit
   vector tangent to the circle $C_{t}$ (or equivalently, orthogonal
   to the direction $\tilde{\gamma}_{t}'(v)$ of the geodesic) at the
   point $\tilde{\gamma}_{t}(v)$.  Because the curvature of the
   Riemannian metric ${\varrho}$ is everywhere negative, the
   infinitesimal expansion rate $\zeta_{t}(v)$ is strictly
   positive. Moreover, because the Riemannian structure is
   $C^{\infty}$, so is the dependence of $\zeta_{t}(v)$ on both $t$
   and $v \in T^{1}\tilde{S}$. Consequently, there exist constants
   $C_{2}<\infty$ and $a>0$ such that for any two unit vectors
   $v_{1}, v_{2}\in T^{1}_{\tilde{y}}\tilde{S}$ and any $t>0$,
   \begin{align*}
     \textrm {distance}
     (\tilde{\gamma}_{t}(v_{1}),\tilde{\gamma}_{t}(v_{2})) &\leq 3\delta
     \quad \Longrightarrow \\
     \textrm {distance}
     (\tilde{\gamma}_{t-s}(v_{1}),\tilde{\gamma}_{t-s}(v_{2})) &\leq C_{2}
     e^{-as} \quad \textrm {for all} \; 0 \leq s \leq t.
   \end{align*}
   Therefore, the integral formula for the expansion rate
   $\eta_{t}(v)$ implies that for an appropriate constant $C_{3}<\infty$,
   \begin{gather*}
     \textrm {distance}
     (\tilde{\gamma}_{t}(v_{1}),\tilde{\gamma}_{t}(v_{2})) \leq 1
     \quad \Longrightarrow \\
     C_{3}^{-1}\leq \frac{\eta_{t}(v_{1})}{\eta_{t-s}(v_{2})} \leq
     C_{3} \qquad \textrm {for all} \; t>3\delta \; \textrm {and}\; 0\leq
     s \leq 3\delta.
   \end{gather*}
  
\end{proof}

\end{proof}

\section{Measure of the Crossing Sets}\label{sec:asymptotics}

Fix $x\in S$ and $\alpha >0$.  Let $A,B$ be any two disjoint closed
arcs, each with nonempty interior, on the boundary of the ball
$B (0,\alpha)$ in the unit tangent space $T^1_{x}S$. Recall that we
have agreed to identify the arcs $A,B$ with their images $A^{T},B^{T}$
on the closed curve $\partial D(x, \alpha T^{-1})$ under the
scaled exponential mapping $v\mapsto \exp _{x}(v/T)$.  Recall also
that $\Sigma(A,B;T)$ is the set of all sequences $\omega \in \Sigma$
such that the intersection of the geodesic segment
$p\circ\pi (\mathcal{F}_{\omega})$ with the neighborhood
$D(x, \alpha T^{-1})$ is a geodesic segment that either
coincides with or extends to a geodesic crossing of
$D(x, \alpha T^{-1})$ from boundary arc $A^{T}$ to boundary
arc $B^{T}$.

\begin{proposition}\label{proposition:measure-asymptotics}
For all $A,B,$ and $\alpha >0$,
\begin{equation}\label{eq:measure-asymptotics-again}
	\lim_{T \rightarrow \infty} T\lambda (\Sigma (A,B;T))=
	\frac{1}{2} \kappa \beta_{A,B}E_{\lambda}F
\end{equation}
where $\beta_{A,B}$ is as defined by equation~\eqref{eq:beta} (with
$D=B(0, \alpha)$) and $\kappa=1/\textrm {area}(S)$.
\end{proposition}

\begin{proof}
  Be definition, if $\omega \in \Sigma(A,B;T)$ then there exist unique
  times $s_{A}(\omega)<s_{B}(\omega)<F(\omega)$ such that the segment
  $\phi ([s_{A}(\omega),s_{B}(\omega)])$ projects via $p\circ\pi$ to a
  geodesic crossing of $D(x,\alpha T^{-1})$ from boundary arc $A^{T}$ to
  boundary arc $B^{T}$. It is possible that $s_{A}(\omega)<0$; this will
  occur if and only if  $p\circ\pi (\omega,0)$ is an interior  point
  of  $D(x, \alpha T^{-1})$. Now the surface area of $D(x,
  \alpha 
  T^{-1})$ is of order $T^{-2}$; hence, since $\lambda^{*}$ is the
  pullback of  the normalized Liouville measure $\mu_{L}$, the
  $\lambda-$measure of the set of $\omega \in  \Sigma(A,B;T)$  such
  that $s_{A}(\omega)<0$ is also of order $T^{-2}$. Consequently, in
  proving \eqref{eq:measure-asymptotics-again} we may ignore the
  contribution of the set
  \begin{displaymath}
    \Sigma(A,B;T)_{-}=\xset {\omega \in \Sigma(A,B;T)\,:\,
    s_{A}(\omega)<0}.
\end{displaymath}
Set $\Sigma(A,B;T)_{+}=\Sigma(A,B;T)\setminus \Sigma(A,B;T)_{-}$.

Denote by $\Upsilon (A,B;T)$ the set of all $u\in T^1S$ that are
tangents to geodesic segments from arc $A^{T}$ to arc $B^{T}$.  This
set nearly coincides with the projections of those points $(\omega ,s)\in
\Sigma_{F}$ such that $\omega \in \Sigma (A,B;T)$ and $0\leq s_{A} (\omega
)<s<s_{B} (\omega)$, the difference being accounted for by the set $
\Sigma(A,B;T)_{-}$. Consequently, by 
equation~\eqref{eq:measure-transfer},
\begin{align*}
  \lambda (\Sigma (A,B;T))&= \int_{\Sigma (A,B;T)}
                            \frac{s_{B}-s_{A}}{s_{B}-s_{A}} \,d\lambda \\
                          &=\int_{\pi^{-1}\Upsilon (A,B;T)} \frac{1}{s_{B} (\omega)-s_{A}
                            (\omega)} \,d\lambda^{*} (\omega ,s) \times \int_{\Sigma}F\,d\lambda +O(T^{-2})\\
                          &=\int_{\Upsilon (A,B;T)} \frac{1}{\tau
                            (u)} \,d\mu_{L} (u) \times
                            \int_{\Sigma}F\,d\lambda +O(T^{-2}), 
\end{align*}
where $\tau(u)$ is the length of the geodesic segment from $A$ to $B$
on which $u$ lies, for any $u\in \Upsilon (A,B;T)$.

Now we exploit the defining property of the Liouville measure $\mu_{L}$,
specifically, that locally $\mu_{L}$ is the product of normalized
surface area with the Haar measure on the circle.  For large $T$,
the exponential mapping $v\mapsto \exp \{v/T \}$ maps the ball $B
(0,\alpha)$ in the tangent space $T_x S$ onto $D (x;\alpha T^{-1})$
nearly isometrically (after scaling by the factor $T^{-1}$), so
rescaled surface area on $D (x;\alpha T^{-1})$ is nearly identical with the
pushforward of Lebesgue measure on $B (0,\alpha)$, scaled by $T^{-2}$.
Furthermore, the inverse images of geodesic segments across $D
(x;\alpha T^{-1})$ are nearly straight line segments crossing $B
(0;\alpha)$; those that cross from arc $A^{T}$ to arc $B^{T}$ in $\partial \mathcal{D}
(x;\alpha T^{-1})$ will pull back to straight line segments from arc
$A$ to arc $B$ in $\partial B (0;\alpha )$. These can be parametrized
by the angle at which they meet the $x-$axis, as in
Figure~\ref{fig:Strip at angle theta}; for each angle $\theta$, the
integral of $1/\text{length}$ over the region in $B (0;\alpha)$ swept
out by line segments crossing from arc $A$ to arc $B$ at angle
$\theta$ is $\psi (\theta)$, as in Figure~\ref{fig:Strip at angle
theta} (where the convex region is now $B (0;\alpha)$).  Therefore, as
$T \rightarrow \infty$,
\[
	\int_{\Upsilon (A,B;T)} \frac{1}{\tau  (u)} \,d\mu_{L} (u)\sim T^{-1}
	\frac{1}{2\pi\, \text{area} (S)} \int_{-\pi  /2}^{\pi /2}
	 \psi (\theta ) \,d\theta =T^{-1} \kappa \beta_{A,B}/2.
\]
\end{proof}

Similar calculations can be used to show  there is vanishingly
small  probability that one
of the first $n$ geodesic segments $p\circ\pi (\mathcal{F}_{\sigma^{i}\omega})$
will hit both $D(x;\alpha T^{-1})$ and $D(x';\alpha T^{-1})$, where
$x\not =x'$ are distinct point of $S$. Define 
$H (x,x';\alpha T^{-1})$ to be the set of all $\omega \in \Sigma$ such
that the vertical fiber $\mathcal{F}_{\omega}$ over $(\omega ,0)$ in
$\Sigma_{F}$ projects to a geodesic segment that intersects both 
 $D(x;\alpha T^{-1})$ and $D(x';\alpha T^{-1})$.

\begin{proposition}\label{proposition:double-hits}
For any two distinct points $x,x'\in S$ and each $\alpha >0$,
\[
	\lim_{T \rightarrow \infty} T\lambda (H (x,x';\alpha T^{-1}))=0.
\]
\end{proposition}

\begin{proof}
  Assume that $T$ is sufficiently large that the closed disks
  $\bar{\mathcal{D}} (x;2\alpha T^{-1})$ and $\bar{\mathcal{D}} (x';2\alpha T^{-1})$ do
  not intersect.  For any $\omega$ such that the fiber
  $\mathcal{F}_{\omega}$ projects to a geodesic segment that enters
  $D(x;2\alpha T^{-1})$ there will be unique times
  $0<s_{0} (\omega)<s_{1} (\omega)<F (\omega)$ of entry and exit
  (except, as in the proof of
  Proposition~\ref{proposition:measure-asymptotics}, for a set of size
  $O(T^{-2})$); for those $\omega$ such that the projection of
  $\mathcal{F}_{\omega }$ enters the smaller disk
  $D(x;\alpha T^{-1})$, the sojourn time
  $s_{1} (\omega)-s_{0} (\omega)$ will be at least $\alpha T^{-1}$.

  Denote by $\Upsilon (x,x';\alpha T^{-1})$ the set of all tangent
  vectors $u\in T^1S$ based at points in
  $D(x;\alpha T^{-1})$ such that $u$ lies on the directed
  geodesic segment $\pi (\mathcal{F}_{\sigma^{i}\omega})$ for some
  sequence $\omega \in H (x,x';\alpha T^{-1})$. Since $x$ and $x'$ are
  distinct points of $S$, the neighborhoods
  $D(x;\alpha T^{-1})$ and $D(x';\alpha T^{-1})$
  are separated by at least $\text{dist} (x,x')/2$ (for large $T$), so
  there is a constant $C=C (x,x',\alpha) <\infty$ such that for every
  point $y\in D(x,\alpha T^{-1})$ the set of angles
  $\theta$ such that $(y,\theta)\in \Upsilon (x,x';\alpha T^{-1})$ has
  Lebesgue measure less than $CT^{-1}$. Now
\begin{align*}
	\lambda (H (x,x';\alpha T^{-1}))&=\int_{H (x,x';\alpha T^{-1})}
	\frac{s_{1}-s_{0}}{s_{1}-s_{0}} \,d\lambda \\
	&= \int_{\pi^{-1}\Upsilon (x,x';\alpha T^{-1})} \frac{1}{s_{1}
	(\omega)-s_{0} (\omega)} \,d\lambda^{*} (\omega ,s)\times E_{\lambda}F\\
	&=\int_{\Upsilon (x,x';\alpha T^{-1})} \frac{1}{\tau (u)}
	\,dL(u) \times E_{\lambda}F \\
	&\leq  T^{-1} L (\Upsilon (x,x';\alpha T^{-1}))E_{\lambda F}
\end{align*}
where $\tau (u)$ is the crossing time of $D(x;2\alpha T^{-1})$ by the
geodesic with initial tangent vector $u$. 
Using once again the fact that (normalized) Liouville measure is the
product of normalized hyperbolic area with Lebesgue angular measure,
we see that for a suitable constant $C'<\infty$,
\[
	L (\Upsilon (x,x';\alpha T^{-1}))\leq C' (\alpha T^{-1})^{2} \times CT^{-1};
\]
thus, $\lambda (H (x,x';\alpha T^{-1}))=O (T^{-2})$.
\end{proof}

\section{Decomposition of the Events
  $N_{A,B}=k$}\label{sec:decomposition} 

In this section we show that the events
$\{\omega \,:\, N_{A,B} (\omega ) =k\}$ can be approximated by sets
consisting of those sequences $\omega \in \Sigma$ whose first $n$
letters contain exactly $k$ occurrence of certain ``magic subwords''
each of length $m= (\log n)^{2} (\approx \log T)^{2}$.  As in
section~\ref{sec:asymptotics}, let $A,B$ be any two disjoint closed
arcs, each with nonempty interior, on the boundary of the ball
$B (0,\alpha)$ in the unit tangent space $T_x S$.  We identify the
neighborhood $D(x,\alpha T^{-1})$ in $S (=\mathcal{P})$
with the ball $B(0,\alpha )$ via the scaled exponential mapping
(cf. equation~\eqref{eq:alphaNH}), so the arcs $A,B$ are identified with arcs in
$\partial D(x,\alpha T^{-1})$, denoted by $A^{T},B^{T}$,
whose arc-lengths are roughly proportional to $T^{-1}$.  Recall that
$N_{A,B} (\omega)$ is the number of crossings of the neighborhood
$D(\omega,\alpha T^{-1}) $ from boundary arc $A^{T}$ to
boundary arc $B^{T}$ by the geodesic segment
\begin{displaymath}
  p\circ\pi \left (\cup_{i=0}^{n-1}\mathcal{F}_{\sigma^{i}\omega}\right),
\end{displaymath}
where $n=[T/E_{\lambda}F]$, $\pi$ is the map from the suspension space $\Sigma_F$ to $T^1(S)$, and $p$ is the natural projection down from $T^1(S)$ to $S$. Recall also (section~\ref{sec:heuristics})
that $\Sigma (A,B;T)$ is the set of all sequences $\omega \in \Sigma$
such that the geodesic segment $p\circ\pi(\mathcal{F}_{\omega})$
intersects $D(x, \alpha T^{-1})$ in a geodesic segment with
terminal endpoint in the boundary arc $B^{T}$ that extends to a
geodesic crossing from boundary arc $A^{T}$ to boundary arc $B^{T}$.

\begin{lemma}\label{lemma:n=k}
  The set $\xset {\omega\,:\, N_{A,B}=k}$ differs from the set
  \begin{displaymath}
    \xset {\omega \in \Sigma \,:\,
      \sum_{i=0}^{n-1}\mathbf{1}_{\Sigma(A,B;T)}(\sigma^{i}\omega)=k }
  \end{displaymath}
  by a set of $\lambda-$measure tending to $0$ as $T \rightarrow \infty$.
\end{lemma}

\begin{proof}
  The symmetric difference of the two sets is contained in the set of
  all $\omega\in\Sigma$ such that either (a) $\omega \in
  \Sigma(A,B;T)$ or $\sigma^{n-1}\omega \in \Sigma(A,B;T)$, or (b) at
  least one of the geodesic segments $p\circ \pi
  (\mathcal{F}_{\sigma^{i}\omega})$, where $0 \leq i < n$, makes more
  than one visit to the neighborhood $D(x; \alpha T^{-1})$.
  Propositions~\ref{proposition:no-quick-entries} and
  \ref{proposition:no-quick-returns} ensure that this set has
  measure $\rightarrow 0$ as $T \rightarrow \infty$.
\end{proof}

Next, recall that the sequence space $\Sigma$ is equipped with the
metric $d(\omega,\omega')= e^{-n(\omega,\omega')}$, where
$n(\omega,\omega')$ is the minimum nonnegative integer $j$ such that
the sequences $\omega,\omega'$ differ in the $\pm j$ entry, and that
the suspension space $\Sigma_{F}$ inherits from $d$ an induced
``taxicab'' metric satisfying the inequality
\eqref{eq:suspFlowHolder}. Cylinder sets $\Sigma_{m}(\omega)$ are open
balls in $\Sigma$ relative to the metric $d$
(cf. equation~\eqref{eq:cylinder}). Since
the semi-conjugacy  $\pi: \Sigma_{F}\rightarrow T^{1}S$ is H\"{o}lder,
it follows that if $\omega, \omega' \in \Sigma$ are at distance $<
\varepsilon$, then the geodesics $p\circ\pi (\phi_{t}(\omega))$ and
$p\circ\pi (\phi_{t}(\omega'))$ remain at distance $< e^{\max F}\varepsilon$
for all $|t|\leq \max_{\Sigma}F$. Thus, if one of the geodesic segments
crosses from arc $A^{T}$ to $B^{T}$ without coming sufficiently near
one of the endpoints of either $A^{T}$ or $B^{T}$, then so will the other;
and similarly, if one  stays sufficiently far
away from the arcs $A^{T},B^{T}$ then so will the other. In
particular, if
\begin{equation}\label{eq:m:n}
	 m= (\log n)^{2} (\approx (\log T)^{2}),
\end{equation}
then for every $\omega' \in \Sigma_{m}(\omega)$ the geodesic segments
$p\circ\pi (\phi_{t}(\omega))_{|t| \leq \max F}$ and
$p\circ\pi (\phi_{t}(\omega'))_{|t| \leq \max F}$ remain at distance
less than $n^{-C \log n}$ for a suitable constant $C>0$,  and hence,
for every $\omega \in \Sigma$ one of the following will hold:
\begin{enumerate}
\item [(i)] $\Sigma_{m}(\omega) \subset \Sigma(A,B;T)$; 
\item [(ii)] $\Sigma_{m}(\omega)\subset \Sigma(A,B;T)^{c}$; or
\item [(iii)] for every $\omega' \in\Sigma_{m}(\omega)$ the geodesic
  segment $p\circ\pi (\phi_{t}(\omega'))_{|t| \leq \max F}$  will pass
  within distance $C' n^{-C \log n}$ of one of the endpoints of 
  arc $A^{T}$ or arc $B^{T}$.
\end{enumerate}

%

\begin{proposition}\label{proposition:decomposition}
For each pair $A,B$ of non-overlapping closed arcs of $\partial B
(0,\alpha)$ and each $T\geq 1$ there exist sets finite subsets
$\mathcal{J}_{1}\subset \mathcal{J}_{2}$ 
of $\Sigma$ such that

  \begin{enumerate}
\item [(A)] for each 
$\omega \in\mathcal{J}_{1}$ the cylinder set $\Sigma_{m}(\omega)$ is of type (i);
\item [(B)]  for each 
$\omega \not\in\mathcal{J}_{2}$ the cylinder set $\Sigma_{m}(\omega)$ is of type (ii); and
\item [(C)] the set 
  $\cup_{ \omega \in \mathcal{J}_{2}\setminus\mathcal{J}_{1}}\Sigma_{m}(\omega)$ has
  $\lambda-$measure less than $o(n^{-r})$ for all $r>0$.
\end{enumerate}

\end{proposition}

\begin{proof}
The sets  $\mathcal{J}_{1}$  and $\mathcal{J}_{2}\setminus
\mathcal{J}_{1}$ are gotten by selecting representatives of each
cylinder $\Sigma_{m}(\omega)$ of type (i)
and  type (iii), respectively. What must be proved is assertion (C). 

By construction, for every $\omega '$ not in $\cup_{ \omega \in
  \mathcal{J}_{2}}\Sigma_{m}(\omega)$ the geodesic segment $p\circ\pi
(\phi_{t}(\omega,0))_{t \leq \max F}$ must pass within distance $C'n^{-C
  \log n}=C'' T ^{-C''' \log T}$ of one of the four endpoints of
$A^{T}$ or $B^{T}$.  Proposition~\ref{proposition:no-quick-entries}
implies that the normalized Liouville measure of the set of geodesic
rays that enter one of these four regions  by time $\max F$ is of
order $O(T ^{-C''' \log T})=O(n^{-C \log n})$. 
\end{proof}
 
\begin{definition}\label{definition:magicSubwords}
Given arcs $A,B$ as in Proposition \ref{proposition:decomposition} and
$T\geq 1$, define the \emph{magic subwords} for the triple $(A,B;T)$
to be the words
\begin{displaymath}
  \omega_{-m}\omega_{-m+1}\cdots \omega_{m}
\end{displaymath}
where $\omega \in \mathcal{J}_{1}$.
\end{definition}

\begin{corollary}\label{corollary:magicSubwordCounts}
The symmetric difference between the sets $\{\omega \,:\, N_{A,B}=k\}$
and the set of $\omega \in \Sigma$ with exactly $k$ occurrences of one
of the magic subwords in the segment $\omega_{1}\omega_{2}\dotsb
\omega_{n}$ has $\lambda -$measure  $\rightarrow 0$ as $T \rightarrow\infty$.
\end{corollary}

\begin{remark}\label{remark:magicSubwords}
The set of magic subwords for a particular value of $T$ will in
general have no clear relationship to the magic subwords for a
different value of $T$. 
\end{remark}

\begin{proposition}\label{proposition:magicWordMeasure}
For each $T$ let $\mathcal{M}=\mathcal{M}_{T}$ be the set of magic
subwords for a fixed pair $A,B$ of boundary arcs and fixed $\alpha
>0$.  Then 
\begin{equation}\label{eq:magicWordMeasure}
	\lim_{T \rightarrow \infty }T\lambda \{\omega \,:\,
	(\omega_{-m}\omega_{-m+1}\dotsb \omega_{m})\in \mathcal{M}\} 
	= \frac{1}{2}  \kappa \beta_{A,B}E_{\lambda}F
\end{equation}
where $\beta_{A,B}$ is defined by equation \eqref{eq:beta}.
\end{proposition}

\begin{proof}
This follows directly from
Propositions~\ref{proposition:decomposition} and
\ref{proposition:measure-asymptotics}. 
\end{proof}

\begin{corollary}\label{corollary:quickReturns}
If $\omega \in \Sigma$ is chosen randomly according to $\lambda$, then
the probability that the initial segment $\omega_{1}\omega_{2}\dotsb
\omega_{[\\log T]^{\kappa}}$ contains a magic subword converges to $0$
as $T \rightarrow \infty$. Similarly, the probability that the
segment $\omega_{1}\omega_{2}\dotsb
\omega_{n}$ contains magic subwords separated by fewer than $(\log
n)^{\kappa }$ letters converges to zero as $n \rightarrow \infty$. 
\end{corollary}

\begin{proof}
  This follows directly from
  Propositions~\ref{proposition:no-quick-entries} and
  \ref{proposition:no-quick-returns}.
\end{proof}

\section{Proof of
Propositions~\ref{proposition:reformulation}--\ref{proposition:reformulation-pairs}}   
\label{sec:final-step} 

\begin{proof}
[Proof of \eqref{eq:reformulation} for $r= 1$]
Consider first the case $r=1$. In this case we are given a single pair
$(A,B)$ of  non-overlapping boundary arcs of $\partial B
(0,\alpha)$; we must show that for any integer $k\geq 0$,
\begin{equation}\label{eq:obj-r1}
	\lim_{n \rightarrow \infty} \lambda \{\omega \in \Sigma  \,:\,
	N_{A,B} (\omega)=k\} = \frac{( \kappa   \beta_{A,B}/2)^{k}}{k!}
	e^{- \kappa \beta_{A,B}/2},
\end{equation}
where $\beta_{A,B}$ is defined by equation~\eqref{eq:beta}. Recall
that $N_{A,B} (\omega)$ is the number of geodesic segments in the
collection $I_{n} (\omega)$ that cross the target disk $D(x;\alpha
T^{-1})$ from arc $A$ to arc $B$. By
Corollary~\ref{corollary:magicSubwordCounts}, $N_{A,B} (\omega)$ is
well-approximated by the number $N'_{A,B}$ of magic subwords in the
word $\omega_{1}\omega_{2}\dotsb \omega_{n}$; in particular, for any
$k\geq 0$, the symmetric difference between the events $\{N_{A,B}=k
\}$ and $\{N'_{A,B} =k\}$ has $\lambda -$measure tending to
$0$. Consequently, it suffices to prove that \eqref{eq:obj-r1} holds
when $N_{A,B}$ is replaced by $N'_{A,B}$.

Recall (sec.~\ref{ssec:regeneration}) that any Gibbs process is the
natural projection of a list  process. Thus, on some probability space
there exists a sequence $W_{1},W_{2},W_{3},\dotsc$ of independent
random words of random lengths $\tau_{i}$, such that the infinite
sequence obtained by concatenating  $W_{1},W_{2},W_{3},\dotsc$ has
distribution $\lambda$, that is, for any Borel subset $B$ of
$\Sigma^{+}$,
\[
	P\{W_{1}\cdot W_{2} \cdot W_{3} \dotsb  \in B\}=\lambda (B).
\]
All but the first word $W_{1}$ have the same distribution, and the
lengths $\tau_{i}$ have exponentially decaying tails (cf. inequality
\eqref{eq:exp-regeneration}).  Since the magic subwords are of length
$[\log n]^{2}$, any occurrence of one will typically straddle a large
number of consecutive words in the sequence $W_{i}$. Thus, to
enumerate occurrences of magic subwords, we shall break the sequence
$\{W_{i} \}_{i\geq 1}$ into blocks of length $m=[\log n]^{3}$, and
count magic subwords  block by block. Set
\begin{align*}
	\tilde{W}_{1}&=W_{1}W_{2}\dotsb W_{m},\\
	\tilde{W}_{2}&=W_{m+1}W_{m+2}\dotsb W_{2m},\\
	\tilde{W}_{3}&=W_{2m+1}W_{2m+2}\dotsb W_{3m},
\end{align*}
etc., and denote by $\tilde{\tau}_{k}=\sum_{i=mk-k+1}^{mk}\tau_{i}$
the length (in letters) of the word $\tilde{W}_{k}$.

\begin{claim}\label{claim:no-long-blocks}
For each $C>E\tau_{2}$, there exists $\Lambda (C)>0$ such that for any integer $k\geq 1$
\begin{equation}\label{eq:LD-1}
	P\left\{ \sum_{i=1}^{k} \tau_{i} \geq Ck\right\}\leq e^{-k\Lambda (C)},
\end{equation}
and for all sufficiently large  $C<\infty$, 
\begin{equation}\label{eq:LD-2}
	\lim_{n \rightarrow \infty}P\{\max_{k\leq n} \tilde{\tau}_{k}\geq Cm\}=0.
\end{equation}
The function $C\mapsto \Lambda (C)$ is convex and satisfies
$\liminf_{C \rightarrow \infty}\Lambda (C)/C>0$.
\end{claim}

\begin{proof}
[Proof of Claim~\ref{claim:no-long-blocks}] These estimates follow from the
exponential tail decay property \eqref{eq:exp-regeneration} by
standard results in the elementary large deviations theory, in
particular, \emph{Cram\'{e}r's theorem} (cf. \cite{dembo-zeitouni},
sec.~2.2) for sums of independent, identically distributed random
variables with exponentially decaying tails.  The block lengths
$\tilde{\tau}_{k}$ are gotten by summing the lengths $\tau_{i}$ of
their $m$ constituent words $W_{i}$; for all but the first block
$\tilde{W}_{1}$, these lengths are i.i.d. and satisfy
\eqref{eq:exp-regeneration}. Hence,
Cram\'{e}r's theorem guarantees\footnote{The length of the initial block has a different
distribution than the subsequent blocks, because the first excursion
of the list process has a different law than the rest. However, the
length of the first excursion also has an exponentially decaying tail,
by Proposition~\ref{proposition:regeneration}, so the upper bounds
given by Cram\'{e}r's theorem still apply.}  the existence of a convex rate
function $C\mapsto \Lambda (C)$ and constants
$C'=C' (C) <\infty$ such that inequality
\eqref{eq:LD-1} holds for all $k\geq 1$. Applying this inequality 
with $k=m=[\log n]^{3}$ yields
\[
	P\{\sum_{i=2}^{m+1} \tau_{i}\geq Cm\}\leq e^{-m \Lambda
	(C)}=n^{-3\Lambda (C)} .
\]
Cram\'{e}r's theorem also implies that $\Lambda (C)$ grows at least
linearly in $C$, so by taking $C$ sufficiently large we can ensure
that $\Lambda (C)\geq 2/3$, which makes the probability above smaller
than $n^{-2}$. Since there are only $n$ blocks, it follows that the
probability that $\tilde{\tau}_{k}\geq Cm$ for one of them is smaller
than $n^{-1}$.
\end{proof}

\begin{claim}\label{claim:first-block}
The probability that a magic subword occurs in the concatenation of
the first two blocks $\tilde{W}_{1}\tilde{W}_{2}$ converges to $0$ as
$T \rightarrow \infty$.
\end{claim}

\begin{proof}
[Proof of Claim~\ref{claim:first-block}]
The event that one of the first two blocks has length $\geq C[\log
n]^{3}$ can be ignored, by Claim~\ref{claim:no-long-blocks}. On the
complementary event, an occurrence of a magic subword in
$\tilde{W}_{1}\tilde{W}_{2}$ would require that the magic subword
occurs in the first $2C[\log n]^{3}$ letters. By
Corollary~\ref{corollary:quickReturns}, the probability of this event
tends to $0$ as $T \rightarrow \infty$.
\end{proof}

It follows from Claim~\ref{claim:no-long-blocks} and
Corollary~\ref{corollary:quickReturns} that with probability tending
to $1$ as $n \rightarrow \infty$, no block $\tilde{W}_{k}$ among the
first $n$ will contain more than one magic subword. On this event,
then, the number $N_{A,B}$ of magic subwords that occur in the first
$n$ letters can be obtained by counting the number of blocks
$\tilde{W}_{k}$ that contain magic subwords and then adding the number
of magic subwords that straddle two consecutive blocks.

\begin{claim}\label{claim:no-straddles}
As $n \rightarrow \infty$, the probability that a magic subword
straddles two consecutive blocks $\tilde{W}_{k},\tilde{W}_{k+1}$ among
the first $n/[\log n]^{3}$ blocks converges to $0$.
\end{claim}

\begin{proof}
[Proof of Claim~\ref{claim:no-straddles}] A magic subword, since it
has length $[\log n]^{2}$, can only straddle consecutive blocks
$\tilde{W}_{k},\tilde{W}_{k+1}$ if it begins in one of the last $[\log
n]^{2}$ word $W_{i}$ of the $m=[\log n]^{3}$ words that constitute
$\tilde{W}_{k}$. The words $W_{i}$ are i.i.d. (except for $W_{1}$, and
by Claim~\ref{claim:first-block} we can ignore the possibility that a
magic subword begins in $\tilde{W}_{1}\tilde{W_{2}}$), so the
probability that a magic subword begins in $W_{i}$ does not depend on
$i$. Since only $[\log n]^{2}$ of the $[\log n]^{3}$ words in each
block $\tilde{W}_{k}$ would produce straddles, it follows that the
\emph{expected} number of magic subwords in
$\tilde{W}_{1}\tilde{W}_{2}\dotsb \tilde{W_{n/m}}$ is at least $[\log
n ]$ times the probability that a magic subword straddles two
consecutive blocks. Therefore, the claim will follow if we can show
that the expected number of magic subwords in
$\tilde{W}_{1}\tilde{W}_{2}\dotsb \tilde{W}_{n/m}$ remains bounded as
$T \rightarrow \infty$. Denote the number of such magic subwords by
$N''_{A,B}$.

The number of letters in the concatenation
$\tilde{W}_{1}\tilde{W}_{2}\dotsb \tilde{W}_{n/m}$ is
$\sum_{i=1}^{n}\tau_{i}$, which by Claim~\ref{claim:no-long-blocks}
obeys the large deviation bound \eqref{eq:LD-1}. Fix $K<\infty$, and let $G$ be the event
that $\sum_{i=1}^{n}\tau_{i} \leq nK$. On this event, $N''_{A,B}$ is
bounded by the number of magic subwords in the first $nK$ letters of
the concatenation $W_{1}W_{2}\dotsb$. Since the concatenation
$W_{1}W_{2}\dotsb $ is, by Proposition~\ref{proposition:regeneration},
a version of the Gibbs process associated with the Gibbs state
$\lambda$, which by shift-invariance is stationary, it follows that
the expected number of magic subwords in the first $nK$ letters is 
$nK \times $ the probability that a magic subword begins at the very
first letter of $W_{1}W_{2}\dotsb $. But by
Proposition~\ref{proposition:measure-asymptotics}, this probability is
asymptotic to $ T^{-1}\alpha \beta_{A,B}E_{\lambda} F$; thus, for
large $T$,
\begin{equation*}
	EN''_{A,B}\mathbf{1}_{G}\leq nKT^{-1}\alpha
	\beta_{A,B}E_{\lambda}F =K\alpha \beta_{A,B}. 
\end{equation*}

It remains to bound the contribution to the expectation from the
complementary event $G^{c}$. For this, we use the large deviation
bound \eqref{eq:LD-1}. On the event that $\sum_{i=1}^{n}\tau_{i} \leq
n (K+k)$, the count $N''_{A,B}$ cannot be more than $n (K+k)$; hence,
\[
	EN''_{A,B}\mathbf{1}_{G^{c}}\leq \sum_{k=1}^{\infty} n (K+k)
	e^{-n\Lambda (K+k)} .
\]
Since $\Lambda (C)$ grows at least linearly in $C$, this sum remains
bounded provided $K$ is sufficiently large.
\end{proof}

Recall that $N'_{A,B}$ is the number of magic subwords in the first
$n$ letters of the sequence $\tilde{W}_{1}\tilde{W}_{2}\dotsb$
obtained by concatenating the words in the regenerative
representation. The blocks $\tilde{W}_{k}$ are independent, and except
for the first all have  the same distribution, with common mean length
$mE\tau_{2}$. Let $N^{*}_{A,B}$ be the number of magic subwords in the
segment $\tilde{W}_{2}\tilde{W_{3}}\dotsb \tilde{W_{\nu}}$, where 
$\nu =\nu (n)=n/[m E\tau_{2}]$. By the central limit theorem, with probability
approaching $1$  the
length $\sum_{i=1}^{\nu m}\tau_{i}$ of the segment
$\tilde{W}_{2}\tilde{W_{3}}\dotsb \tilde{W_{\nu}}$ differs by no more
than $\sqrt{n}\log n$ from $n$, and by the same argument as in the
proof of Claim~\ref{claim:no-straddles}, the probability that a magic
subword occur within the stretch of $2\sqrt{n}\log n$ letters
surrounding the $n$th letter converges to $0$. Thus, as $T \rightarrow
\infty$,
\[
	P\{N'_{A,B}\not =N^{*}_{A,B} \} \longrightarrow 0.
\]
By Claim~\ref{claim:no-long-blocks}
and Corollary~\ref{corollary:quickReturns}, with probability
approaching $1$ no block $\tilde{W}_{k}$ will contain more than $1$
magic subword, and by Claim~\ref{claim:no-straddles} no magic subword
will straddle two blocks $\tilde{W}_{k},\tilde{W}_{k+1}$. Therefore,
with probability $\rightarrow 1$,
\[
	N'_{A,B}=N^{*}_{A,B}=\sum_{k=2}^{\nu} Y(\tilde{W}_{k}),
\]
where $Y(\tilde{W}_{k})$ is the indicator of the event that the block
$\tilde{W}_{k}$ contains a magic subword. These indicators are
independent, identically distributed Bernoulli random variables; by
Proposition~\ref{proposition:magicWordMeasure},
\[
	EY(\tilde{W}_{k}) mE\tau_{2} \sim T^{-1} \frac{1}{2}  \kappa \beta E_{\lambda}F
\]
and so
\[
	E\sum_{k=2}^{\nu} Y(\tilde{W}_{k}) \longrightarrow \frac{1}{2}
	\kappa
	\beta_{A,B}. 
\]
Now Proposition~\ref{proposition:lsn} implies that for any integer
$J\geq 0$,
\[
	P\left\{\sum_{k=2}^{\nu} Y(\tilde{W}_{k})=J \right\}
	\longrightarrow  \frac{( \kappa
	\beta_{A,B}/2)^{J}}{J!}e^{- \kappa \beta_{A,B}/2}, 
\]
proving \eqref{eq:obj-r1}.
\end{proof}

\begin{proof}
[Proof of \eqref{eq:reformulation} for $r\geq 1$]
(Sketch) In general, given $r \geq 1$, we are given a set $\{A_{i},B_{i}\}$ of pairwise
non-overlapping boundary arcs of $\partial B (0,\alpha)$; we must show
that the counts $N_{A_{i},B_{i}}$ converge \emph{jointly} to
independent Poissons with means $\alpha \beta_{A_{i}B_{i}}$,
respectively. The key to this is that the sets $\mathcal{M}_{i}$ of
magic words for the different pairs  $(A_{i},B_{i})$ are pairwise
disjoint, because the arcs $A_{i},B_{i}$ are non-overlapping (a
geodesic segment crossing of $D (x;\alpha T^{-1})$ has unique entrance
and exit points on $\partial D (x,\alpha T^{-1})$, so at most one of
the pairs $(A_{i},B_{i})$ can contain these).

By the same argument as in the case $r=1$, the counts 
$N_{A_{i},B_{i}}$  can be replaced by the sums 
\[
	N^{*}_{A_{i},B_{i}}=\sum_{k=2}^{\nu} Y_{i} (\tilde{W}_{k})
\]
where $Y_{i}(\tilde{W}_{k})$ is the indicator of the event that the block
$\tilde{W}_{k}$ contains a magic subword for the pair
$A_{i},B_{i}$. Since the sets $\mathcal{M}_{i}$ of magic subwords are
non-overlapping,  the vector of these sums follows a multinomial
distribution; hence, by Proposition~\ref{proposition:mult-lsn}, the vector 
\[
	(N^{*}_{A_{i},B_{i}})_{1\leq i\leq r}
\]
converges in distribution to the product of $r$ Poisson distributions,
with means $\frac{1}{2} \kappa\alpha \beta_{A_{i}B_{i}}$.
\end{proof}

\begin{proof}
[Proof of Proposition \ref{proposition:reformulation-pairs}]
The argument is virtually the same as that for the case $r\geq 2$ of
Proposition~\ref{proposition:reformulation}; the only new wrinkle is
that the sets $\mathcal{M}_{i}$ and $\mathcal{M}'_{i'}$ of magic words
for the pairs $(A_{i},B_{i})$ and $(A'_{i'},B'_{i'})$ need not be
disjoint, because it is possible for a geodesic segment across the
fundamental polygon $\mathcal{P}$ to enter both $D (x,\alpha T^{-1})$
and $D (x';\alpha T^{-1})$. However,
Proposition~\ref{proposition:double-hits} implies that the expected
number of such double-hits in the first $n$ crossings of $\mathcal{P}$
converges to $0$ as $T \rightarrow \infty$, and consequently the
probability that there is even one double-hit tends to zero. Thus, the
magic subwords for pairs $A'_{i'},B'_{i'}$ that also occur as magic
subwords for pairs $A_{i},B_{i}$ can be deleted without affecting  the
counts (at least with probability $\rightarrow 1$ as $T \rightarrow
\infty$), and so the counts $N_{A_{i},B_{i}}$ and $N'_{A'_{i'},B'_{i'}}$ may
be replaced by 
\begin{align*}
		N^{*}_{A_{i},B_{i}}&=\sum_{k=2}^{\nu} Y_{i}
		(\tilde{W}_{k}) \quad \text{and}\\
		N^{**}_{A'_{i'},B_{i'}}&=\sum_{k=2}^{\nu} Y'_{i'} (\tilde{W}_{k})
\end{align*}
where $Y_{i}(\tilde{W}_{k})$ and $Y'_{i'}$ are the the indicators of
the events that the block $\tilde{W}_{k}$ contains a magic subword for
the appropriate pair (with deletions of any duplicates). Since the
adjusted sets of magic subwords are non-overlapping, the vector of
these counts $N^{*}_{A_{i},B_{i}}$ and $N^{**}_{A'_{i'},B_{i'}}$
follows a multinomial distribution, and so the convergence
\eqref{eq:reformulation-pairs} holds, by
Proposition~\ref{proposition:mult-lsn}.
\end{proof}

\section{Global Statistics}\label{sec:global-stats}

In this section we show how Theorem~\ref{theorem:global-stats}, which
describes the ``global'' statistics of the tessellation
$\mathcal{T}_{T}$ induced by a random geodesic segment of length $T$,
follows from the ``local'' description provided by
Theorem~\ref{theorem:random-local} and the ergodicity of the Poisson
line process with respect to translations.
Theorem~\ref{theorem:random-local} and
Proposition~\ref{proposition:ergodicity-plp} (cf. also
Corollary~\ref{corollary:poly-freqs}) imply that locally -- in balls
$D (x;\alpha T^{-1})$, where $\alpha$ is large -- the empirical
distributions of polygons, their angles and side lengths (after
scaling by $T$) stabilize as $T \rightarrow \infty$. Since this is
true in neighborhoods of all points $x\in S$, it is natural to expect
that these empirical distributions also converge globally. To prove
this, we must show that in those small regions of $S$ where empirical
distributions behave atypically the counts are not so large as to
disturb the global averages. The key is the following proposition,
which limits the numbers of polygons, edges, and vertices in
$\mathcal{T}_{T}$. 

\begin{proposition}\label{proposition:siCounts}
Let $f=f_{T}, v=v_{T}$, and $e=e_{T}$ be the number of polygons,
vertices and edges in the tessellation $\mathcal{T}_{T}$. With
probability one, as $T \rightarrow  \infty$,
\begin{align}\label{eq:vertices}
	\lim_{T \rightarrow \infty} v_{T}/T^{2}&=\kappa/\pi ,\\
\label{eq:edges}
	\lim_{T \rightarrow \infty} e_{T}/T^{2}&= (2\kappa)/\pi, \quad \text{and} \\
\label{eq:faces}
	\lim_{T \rightarrow \infty} f_{T}/T^{2}&=\kappa/\pi.
\end{align}
Moreover, there exists a (nonrandom) constant $C=C_{S}<\infty$ such
that for \emph{every} tessellation $\mathcal{T}_{T}$ induced by a
geodesic segment of length $T$,
\begin{equation}\label{eq:v+e+f-bound}
	v_{T}+e_{T}+f_{T}\leq CT^{2}.
\end{equation}
\end{proposition}

For the proof we will need to know that multiple intersection points
(points of $S$ that a geodesic ray passes through more than twice) do
not occur in typical geodesics. We have the following:

\begin{lemma}\label{lemma:noTriplePoints}
For almost every unit tangent vector $v\in T^1S$, there are no
multiple intersection points on the geodesic ray $(\gamma_{t} (v))_{t\geq 0}$.
\end{lemma}

\begin{proof} Suppose $v \in T^1S$ gives rise to  triple intersection. Let $\gamma$ denote a lift of the geodesic ray $(\gamma_{t} (v))_{t\geq 0}$ to the universal cover $\mathbb H^2$, we have that there must be deck transformations $A, B$ so that the geodesic rays $A\gamma$ and $B\gamma$ have a triple intersection.  In~\cite{Epstein}, it is shown that the set of such geodesics is a positive codimension subvariety for any fixed $A, B$, and therefore, a set of measure $0$. Taking the (countable) union over all possible pairs $A, B$, we have our result.
\end{proof}

\begin{proof}
[Proof of Proposition~\ref{proposition:siCounts}]
The number $v_{T}$ of vertices is the number of self-intersections of
the random geodesic segment $\gamma_{T}:= (\gamma_{t} (\cdot))_{0\leq
t\leq T}$ (unless one counts the beginning and end points of
$\gamma_{T}$ as vertices, in which case the count is increased by
$2$). It is an easy consequence of Birkhoff's ergodic theorem (see
\cite{lalley:duke}, sec. 2.3 for the argument, but beware that
\cite{lalley:duke} seems to be off by a factor of $4$ in his
calculation of the limit) that the number of self-intersections
satisfies \eqref{eq:vertices}. Following is a brief resume of the
argument.

Fix $\epsilon >0$ small, and partition the segment $\gamma_{T}$ into
non-overlapping geodesic segments $\gamma^{i}_{T}$ of length
$\epsilon$ (if necessary, extend or delete the last segment; this will
not change the self-intersection count by more than $O (T)$). If
$\epsilon$ is smaller than the injectivity radius then 
\begin{equation}\label{eq:vertex-sum}
	v_{T}=\sum \sum_{i\not =j} \mathbf{1} (\gamma^{i}_{T}\cap
	\gamma^{j}_{T}\not =\emptyset) 
\end{equation}
is the number of pairs $(i,j)$ such that $\gamma^{i}_{T}$ and
$\gamma^{j}_{T}$ cross. Birkhoff's theorem implies that for each $i$,
the fraction of indices $j$ such that $\gamma^{j}_{T}$ crosses
$\gamma^{i}_{T}$ converges, as $T \rightarrow \infty$, to the
normalized Liouville measure of that region $R_{\epsilon}$ of $T^1S$
where the geodesic flow will produce a ray that crosses
$\gamma^{i}_{T}$ by time $\epsilon$. This implies that the limit on
the left side of \eqref{eq:vertices} exists. To calculate the limit,
let $\epsilon \rightarrow 0$: if $\epsilon >0$ is small, then for each
angle $\theta$ the set of points $x\in S$ such that $(x,\theta)\in
R_{\epsilon}$ is approximately a rhombus of side $\epsilon$ with
interior angle $\theta$. Integrating the area of this rhombus over
$\theta$, one obtains a sharp asymptotic approximation to the normalized
Liouville measure of $R_{\epsilon}$:
\[
	L (R_{\epsilon})\sim 2\epsilon^{2}\int_{0}^{\pi} \sin \theta
	\,d\theta / (2\pi \,\text{area} (S))=2\epsilon^{2} \kappa/\pi.
\]
Since the number of terms in the sum \eqref{eq:vertex-sum} is
$\frac{1}{2} [T/\epsilon^{2}]$, it follows that
$v_{T}/T^{2}\rightarrow \kappa/\pi$.

The limiting relations \eqref{eq:edges} and \eqref{eq:faces} follow
easily from \eqref{eq:vertices}.  With probability one, the geodesic
segment $\gamma_{T}$ has no multiple intersection points, by
Lemma~\ref{lemma:noTriplePoints}. Consequently, as one traverses the
segment $\gamma_{T}$ from beginning to end, one visits each vertex
twice, and immediately following each such visit encounters a new edge
of $\mathcal{T}_{T}$ (except for the initial edge), so
$e_{T}=2v_{T}\pm 2$, and hence \eqref{eq:edges} follows from
\eqref{eq:vertices}. Finally, by Euler's formula, $v-e+f=-\chi (S)$,
and therefore \eqref{eq:faces} follows from
\eqref{eq:vertices}--\eqref{eq:edges}.

No geodesic ray can intersect itself before time $\varrho $, where
$\varrho $ is the injectivity radius of $S$, so for every geodesic
segment $\gamma_{T}$ to length $T$ the corresponding tessellation must
satisfy $v_{T}\leq T^{2}/ \varrho^{2}$. The inequality
\eqref{eq:v+e+f-bound} now follows by Euler's formula and the relation
$e=v\pm 2$.
\end{proof}

\begin{proof}
[Proof of Theorem \ref{theorem:global-stats}] 
We will prove only the assertion concerning the empirical frequencies
of $k-$gons in the induced tessellation. Similar arguments can be used
to prove that the empirical distributions of scaled side-lengths,
interior angles, etc. converge to the corresponding theoretical
frequencies in a Poisson line process. Denote by $\mathcal{T}_{T}$ the
tessellation of the surface $S$ induced by a random geodesic segment
of length $T$.

We first give a heuristic argument that explains how
Theorem~\ref{theorem:random-local},
Corollary~\ref{corollary:poly-freqs}, and
Proposition~\ref{proposition:siCounts} together imply the convergence
of empirical frequencies. Suppose that, for large $T$, the surface $S$
could be partitioned into non-overlapping regions $R_{i}$ each nearly
isometric, by the scaled exponential mapping from the tangent space
based at its center $x_{i}$, to a square of side $\alpha T^{-1}$. (Of
course this is not possible, because it would violate the fact that
$S$ has non-zero scalar curvature.) The hyperbolic area of $R_{i}$ would
be $\sim \alpha^{2}/T^{2}$, and so the number of squares $R_{i}$ in
the partition would be $\sim T^{2}/ (\alpha^{2}\kappa)$.

Assume that $\alpha$ is sufficiently
large that with probability at least $1-\epsilon$, the absolute errors
in the limiting relations \eqref{eq:poly-freqs},
\eqref{eq:vertex-freqs}, and \eqref{eq:kgonFracs} (for some fixed $k$)
of Corollary~\ref{corollary:poly-freqs} are less than $\epsilon$.  By
Theorem \ref{theorem:random-local}, for any point $x\in S$ and any
$\alpha$, the restriction of the geodesic tessellation
$\mathcal{T}_{T}$ to the disk $D (x,2\alpha T^{-1})$, when pulled back
to the ball $B (0,2\alpha)$ of the tangent space $T_x S$, converges
in distribution, as a line process, to the Poisson line process of
intensity $\kappa$. Since this holds for every $x$, it follows
that for all sufficiently large $T$, with
probability at least $1-2\epsilon$, in all but a fraction $\epsilon$
of the regions $R_{i} $ the counts $V_{T} (R_{i} )$ and
$F_{T} (R_{i} )$ of vertices and faces in the regions $R_{i}
$ (in the tessellation $\mathcal{T}_{T}$) and the fractions
$\Phi_{k,T} (R_{i} )$ of $k-$gons will satisfy
\begin{align}\label{eq:V}
 		|V_{T}(R_{i} )/\alpha^{2} - \kappa^{2}/\pi |&<2\epsilon ,\\
\label{eq:F}
		|F_{T}(R_{i} )/\alpha^{2} - \kappa^{2}/\pi
		|&<2\epsilon , \quad \text{and}\\
\label{eq:Phi}
		|\Phi_{k,T} (R_{i} ) -\phi_{k}|&<2\epsilon .
\end{align}
Call the regions $R_{i} $ where these inequalities hold \emph{good},
and the others \emph{bad}. Since all but and area of size $\epsilon
\times \text{area} (S)$ is covered by good squares $R_{i}$, relations
\eqref{eq:F} and \eqref{eq:faces}  imply that the total number of
faces of $\mathcal{T}_{T}$ in the bad squares satisfies
\[
	\sum_{i \;\text{bad}} F_{T} (R_{i}) \leq 4\epsilon T^{2}
	\times \text{area} (S). 
\]
Consequently, regardless of how skewed the empirical distribution of
faces in the bad regions might be, it cannot affect the overall
fraction of $k-$gons by more than $8\epsilon $. Since $\epsilon >0$
can be made arbitrarily small, it follows from \eqref{eq:Phi} that 
\begin{equation}\label{eq:Phi-lim}
	\lim_{T \rightarrow \infty} \Phi_{k,T} (S)= \phi_{k}.
\end{equation}

To provide a rigorous argument, we must explain how the partition into
``squares'' $R_{i}$ can be modified.  Fix $\delta >0$ small, and let
$\Delta$ be a triangulation of $S$ whose triangles $\tau$ all (a) have
diameters less than $\delta$ and (b) have geodesic edges. If $\delta
>0$ is sufficiently small, the triangles of $\Delta$ will all be
contained in coordinate patches nearly isometric, by the exponential
mapping, to disks $B (0,2\delta)$ in the tangent space $TS_{x_{\tau}}$, where
$x_{\tau}$ is a distinguished point in the interior of $\tau$. In each
such ball $B (0,2\delta)$, use an orthogonal coordinate system to
foliate $B (0,2\delta)$ by lines parallel to the coordinate axes, and
then use the exponential mapping to project these foliations to
foliations of the triangles $\tau$; call these foliations
$\mathcal{F}_{x} (\tau)$ and $\mathcal{F}_{y} (\tau)$. If $\delta >0$
is sufficiently small then the curves in $\mathcal{F}_{x} (\tau)$ will
cross curves in $\mathcal{F}_{y} (\tau)$ at angles $\theta \in
[\frac{\pi}{2}-\epsilon ,\frac{\pi}{2}+\epsilon]$, 
 where $\epsilon >0$ is small.

The foliations $\mathcal{F}_{x} (\tau)$ and $\mathcal{F}_{y} (\tau)$
can now be used as guidelines to partition $\tau$ into regions $R_{i}
(\tau)$ whose boundaries are segments of curves in one or the other of
the foliations. In particular, each boundary $\partial R_{i} (\tau)$
should consist of four segments, two from $\mathcal{F}_{x} (\tau)$ and
two from  $\mathcal{F}_{y} (\tau)$, and each should be of length $\sim
\alpha T^{-1}$; thus, for large $T$ each region $R_{i} (\tau)$ will be
nearly a ``parallelogram'' (more precisely, the image of a
parallelogram in the tangent space $TS_{x_{i} (\tau)}$ at a central
point $x_{i} (\tau)\in R_{i} (\tau )$) whose interior angles are within
$\epsilon$ of $\pi /2$. The collection of all regions $R_{i} (\tau )$,
where $\tau$ ranges over the triangulation $\Delta$, is nearly a
partition of $S$ into rhombi; only at distances $O (\alpha
T^{-1})$ of the boundaries $\partial \tau$ are there overlaps. The
total area in these boundary neighborhoods is $o (1)$ as $T
\rightarrow \infty$.

Corollary~\ref{corollary:poly-freqs}, as stated, applies only to
squares. However, any rhombus $R$ whose interior angles are within
$\epsilon$ of $\pi /2$ can be bracketed by squares $S_{-}\subset
R\subset S_{+}$ in such a way that the area of $S_{+}\setminus S_{-}$
is at most $C\epsilon \,\text{area} (S_{+})$, for some $C<\infty$ not
depending on $\epsilon$. Since Corollary~\ref{corollary:poly-freqs}
applies for each of the bracketing squares, it now follows as in the
heuristic argument above that with probability $\geq 1-C'\epsilon$, in
all but a fraction $C\epsilon$ of the regions $R_{i} (\tau)$ the
inequalities \eqref{eq:V}, \eqref{eq:F},and \eqref{eq:Phi} will hold.
The limiting relation \eqref{eq:Phi-lim} now follows as before.
\end{proof}

\section{Extensions, Generalizations, and Speculations}\label{sec:ext}

\medskip

\textbf{A. Finite-area hyperbolic surfaces with cusps.}  We expect
also that Theorems
\ref{theorem:random-local}--\ref{theorem:global-stats} extend to
finite-area hyperbolic surfaces with cusps. For this, however,
genuinely new arguments would seem to be needed, as our analysis for
the compact case relies heavily on the symbolic dynamics of
Proposition~\ref{proposition:boundaryCorrespondence} and the
regenerative representation of Gibbs states
(Proposition~\ref{proposition:regeneration}). The geodesic flow on the
modular surface has its own very interesting symbolic dynamics
(cf. for example \cite{series:modular} and
\cite{adler-flatto:cross-section-map}), but this uses a countably
infinite alphabet (the natural numbers) rather than a finite
alphabet. At present there seems to be no analogue of the regenerative
representation theorem (Proposition~\ref{proposition:regeneration})
for Gibbs states on sequence spaces with infinite alphabets.

\bigskip \textbf{B. Tessellations by closed geodesics.}  It is known
that statistical regularities of ``random'' geodesics (where the
initial tangent vector is chosen from the maximal-entropy invariant
measure for the geodesic flow) mimic those of typical long closed
geodesics. This correspondence holds for first-order statistics
(cf. \cite{bowen:equidistribution}), but also for second-order
statistics (i.e., ``fluctuations): see \cite{lalley:axiomA},
\cite{lalley:homology}, \cite{lalley:duke}). Thus, it should be
expected that Theorems
\ref{theorem:random-local}--\ref{theorem:global-stats} have analogues
for long closed geodesics. In particular, we conjecture the
following.

\begin{conj}\label{conj:closed}
 Let $S$ be a closed hyperbolic surface, and let $x\in S$ be a fixed
point on $S$. From among all closed geodesics of length $\leq T$
choose one -- call it $\gamma_{T}$ -- at random, and let $A_{T}$ be
the intersection of $\gamma_{T}$ with the ball $D (x;\alpha T^{-1})$.
Then as $T \rightarrow \infty$ the random collection of arcs $A_{T}$
converge in distribution to a Poisson line process on $B (0;\alpha )$
of intensity $\kappa$.
\end{conj}

We do not expect that this will be true on a surface of variable
negative curvature, because the maximal-entropy invariant measure for
the geodesic flow coincides with the Liouville measure only in
constant curvature.

\bigskip \textbf{C. Tessellations by several closed geodesics.}
Given Conjecture \ref{conj:closed}, it is natural to expect that if
two (or more) closed geodesics $\gamma_{T},\gamma '_{T}$ are chosen at
random from among all closed geodesics of length $\leq T$, the
resulting tessellations should be independent. Thus, the intersections
of these tessellations with a ball $D (x,\alpha T^{-1})$ should
converge jointly in law to independent Poisson line processes of
intensity $\kappa$.

\appendix{}
\section{Poisson Line Processes}\label{sec:appendix}

\begin{proof}[Proof of Lemma~\ref{lemma:translation-invariance}]
Rotational invariance is obvious, since the angles $\Theta_{n}$ are
uniformly distributed, so it suffices to establish
invariance by translations along the $x-$axis. To accomplish this, we
will exhibit a sequence $\mathcal{L}_{m}$ of line processes that
converge pointwise to a Poisson line process $\mathcal{L}$, and show
by elementary means that each $\mathcal{L}_{m}$ is translationally
invariant. 

Let $\{(R_{n},\Theta_{n}) \}_{n\in \zz{Z}}$ and $\{\Theta_{n} \}_{n\in
\zz{Z}}$ be the Poisson point process used in the construction
\eqref{eq:point-to-line-mapping} of $\mathcal{L}$.  For each
$m=3,5,7,\dotsc$, let $\mathcal{A}_{m}=\{k\pi / m \}_{0\leq k<m}$ (the
restriction to odd $m$ prevents $\pi /2$ from occurring in $A_{m}$).
For each $n\geq 1$, let $\Theta^{m}_{n}=[m\Theta_{n}]/m$ be the
nearest point in $\mathcal{A}_{m}$ less than $\Theta_{n}$. By
construction, for each $m$ the random variables $\Theta^{m}_{n}$ are
independent and identically distributed, with the uniform distribution
on the finite set $\mathcal{A}_{m}$. Now define $\mathcal{L}_{m}$ to
be the line process constructed in the same manner as $\mathcal{L}$,
but using the discrete random variables $\Theta^{m}_{n}$ instead of
the continuous random variables $\Theta_{n}$.  Clearly, as $m
\rightarrow \infty$ the sequence $\mathcal{L}_{m}$ of line processes
converges to $\mathcal{L}$.

It remains to show that each of the line processes $\mathcal{L}_{m}$
is invariant by translations along the $x-$axis. For this, observe
that for each $\theta_{k}\in \mathcal{A}_{m}$ the thinned process
$\mathcal{R}_{m,k}$ consisting of those $R_{n}$ such that
$\Theta^{m}_{n}=\theta_{k}$ is itself a Poisson point process on
$\zz{R}$ of intensity $\kappa /m$, and that these thinned Poisson
point processes are mutually independent.\footnote{The thinning and
  superposition laws are elementary properties of Poisson point
  processes. The thinning law follows from the superposition property;
  see Kingman~\cite{kingman} for a proof of the latter.}
Consequently, the line process $\mathcal{L}_{m}$ is the superposition
of $m$ independent line processes $\mathcal{L}^{k}_{m}$, with
$k=1,2, \cdots ,m$, where $\mathcal{L}^{k}_{m}$ is the subset of all
lines in $\mathcal{L}^{m}$ that meet the $x-$axis at angle
$\pi /2-\theta_{k}.$ Since the constituent processes
$\mathcal{L}^{k}_{m}$ are independent, it suffices to show that for
each $k$ the line process $\mathcal{L}^{k}_{m}$ is
translation-invariant.  But this is elementary: the points where the
lines in $\mathcal{L}^{k}_{m}$ meet the $x-$axis form a Poisson point
process on the real line, and Poisson point processes on the real line
of constant intensity are translation-invariant.
\end{proof}

\begin{proof} [Proof of Corollary~\ref{corollary:line-intensity}] By
  rotational invariance, it suffices to show this for the
  $x-$axis. Let $\mathcal{L}_{m}$ and $\mathcal{L}^{k}_{m}$ be as in
  the proof of Lemma~\ref{lemma:translation-invariance}; then by an
  easy calculation, the point process of intersections of the lines in
  $\mathcal{L}^{k}_{m}$ with the $x-$axis is a Poisson point process
  of intensity $(\kappa / m ) \sin \theta_{k}$. Summing over $k$ and
  then letting $m \rightarrow \infty$, one arrives at the desired
  conclusion.
\end{proof}

\begin{proof} [Proof of
  Proposition~\ref{proposition:characterization}] The hypothesis that
  $\Gamma$ encloses a strictly convex region guarantees that if a line
  intersects both $A$ and $B$ then it meets each in at most one
  point. Denote by $L_{\{A,B\}}$ the set of all lines that intersect
  both $A$ and $B$. If $A$ and $B$ are partitioned into
  non-overlapping sub-arcs $A_{i}$ and $B_{j}$ then $L_{\{A,B\}}$ is
  the \emph{disjoint} union $\cup_{i,j}L_{\{A_{i},B_{i}\}}$. Since the
  sets $L_{\{A_{i},B_{i}\}}$ are piecewise disjoint, the corresponding
  regions of the strip $\zz{R}\times [0,\pi )$ (in the standard
  parametrization \eqref{eq:point-to-line-mapping}) are
  non-overlapping, and so, by a defining property of the Poisson point
  process $\{(R_{n},\Theta_{n}) \}_{n\not\in \zz{Z}}$, the counts
  $N_{\{A_{i},B_{j}\}}$ are independent Poisson random variables.
  Since the sum of independent Poisson random variables is Poisson, to
  finish the proof it suffices to show that for arcs $A,B$ of length
  $<\varepsilon$ the random variables $N_{\{A,B\}}$ are Poisson, with
  means $\kappa \beta_{A,B}$.

If $\varepsilon >0$ is sufficiently small then any
line $L$ that intersects two boundary arcs $A,B$ of length
$\leq \varepsilon$ must intersect the two straight line segments
$\tilde{A},\tilde{B}$ connecting the endpoints of $A$ and $B$,
respectively; conversely, any line that intersects both
$\tilde{A},\tilde{B}$ will intersect both $A,B$. Therefore, we may
assume that the arcs $A,B,A_{i},B_{i}$ are straight line segments
of length $\leq \varepsilon$. Because Poisson line processes are
rotationally invariant, we may further assume that  $A$ is the
interval $[-\varepsilon /2,\varepsilon /2]\times \{0 \}$.

We now resort once again to the discretization technique used in the
proof of Lemma~\ref{lemma:translation-invariance}. For each
$m=3,5,7,\dotsc$, let $N^{m}_{\{A,B\}}$ be the number of lines in the line
process $\mathcal{L}_{m}$ that cross the segments $A,B$. Clearly,
$N^{m}_{\{A,B\}}\rightarrow N_{\{A,B\}}$ as $m \rightarrow \infty$, so it
suffices to show that for each $m$ the random variable $N^{m}_{\{A,B\}}$
has a Poisson distribution with mean $\mu_{m}\rightarrow \kappa
\beta_{A,B}$. 

Recall that the line process $\mathcal{L}_{m}$ is a superposition of
$m$ independent line processes $\mathcal{L}^{k}_{m}$, and that for
each $k$ the lines in $\mathcal{L}^{k}_{m}$ all meet the $x-$axis at a
fixed angle $|\pi /2-\theta_{k}|$. Hence,
$N^{m}_{\{A,B\}}=\sum_{k}N^{m,k}_{\{A,B\}}$, where $N^{m,k}_{\{A,B\}}$ is the
number of lines in $\mathcal{L}^{k}_{m}$ that cross both $A$ and
$B$. The random variables $N^{m,k}_{\{A,B\}}$ are independent; thus, to
show that $N^{m}_{\{A,B\}}$ has a Poisson distribution it suffices to show
that each $N^{m,k}_{\{A,B\}}$ is Poisson.  By construction, the lines in
$\mathcal{L}^{k}_{m}$ meet the line $(s \cos \theta_{k}, s \sin
\theta_{k})_{s\in \zz{R}}$ at the points of a Poisson point process of
intensity $\kappa /m$; consequently, they meet the $x-$axis at the
points of a Poisson point process of intensity $\kappa |\cos
\theta_{k}|/m$. Now a line that meets the $x-$axis at angle $\theta_{k}$
will cross both $A=[-\varepsilon/2 ,\varepsilon /2]\times \{0 \}$ and $B$ if and
only if its point of intersection with the $x-$axis lies in the
$\theta_{k}-$shadow $J_{k}$ of $B$ on $A$. Therefore, $N^{m,k}_{\{A,B\}}$
has the Poisson distribution with mean $\kappa |J_{k} \cos
\theta_{k}|/m=\kappa \psi_{A,B} (\theta_{k})$. 
It follows that $N^{m}_{\{A,B\}}$ has the Poisson
distribution with mean
\begin{align*}
	EN^{m}_{\{A,B\}}&=m^{-1}\sum_{k=0}^{m-1} \kappa|J_{k}  \cos \theta_{k}|\\
	&=m^{-1}\sum_{k=0}^{m-1} \kappa \psi_{A,B} (\theta_{k})\\
	&\longrightarrow \frac{\kappa }{\pi }  \int_{-\pi /2}^{\pi
	/2}\psi_{A,B} (\theta)\,d\theta . 
\end{align*}

\end{proof}

\begin{proof}
  [Proof of Corollary~\ref{corollary:ppIntensity}]
It suffices to prove this for disks of small radius, because by the
translation-invariance of $\mathcal{L}$,
\[
	EV (D )\sim \frac{1}{\pi
	\varrho^{2}}\int_{D}EV (B (x,\varrho)) \,dx = EV (B
	(0,\varrho)) |D|/ (\pi \varrho^{2})
\]
as $\varrho \rightarrow 0$. Let $\gamma$ be a chord of $B
(0,\varrho)$, and $H_{\gamma}$ the event that $\gamma \in
\mathcal{L}\cap B (0,\varrho)$. Conditional on $H_{\gamma}$, the
number of intersection points on $\gamma$ is Poisson with mean
$2\kappa |\gamma |/\pi $, by Corollary~\ref{corollary:line-intensity}
and Proposition~\ref{proposition:characterization}.\footnote{The event
$H_{\gamma}$ has probability $0$, but it is the limit of the
positive-probability events that $\mathcal{L}$ has a line which
intersects small boundary arcs centered at the endpoints of
$\gamma$. The conditional distribution of $\mathcal{L}$ given
$H_{\gamma}$ can be interpreted as the limit of the conditional
distributions given these approximating events. The independence
assertion of Proposition~\ref{proposition:characterization} guarantees
that, conditional on $H_{\gamma}$, the distribution of
$\mathcal{L}\cap B (0,\varrho)$ is the same as the unconditional
distribution of $(\mathcal{L}\cap B (0,\varrho))\cup \{\gamma \}$.}
Therefore,
\[
	EV (B (0,\varrho))= \frac{1}{2} \frac{2\kappa}{\pi}
	E\left( \sum_{\gamma \in \mathcal{L}\cap B(0,\varrho)} |\gamma
	|\right):=\frac{\kappa}{\pi }E\Psi  
	(\mathcal{L}\cap B (0,\varrho )) .
\]
(The factor of $1/2$ accounts for the fact that each intersection
point lies on two chords.)

The expectation $E\Psi (\mathcal{L}\cap B (\cup ,\varrho))$ is
easily evaluated using the standard construction of the Poisson line
process (Definition~\ref{definition:plp}). The lines of $\mathcal{L}$
that cross  $B (0,\varrho)$ are precisely those corresponding to
points $R_{n}$ such that
$-\varrho <R_{n}<\varrho$. For any such $R_{n}$, the length of the chord $\gamma
=\gamma_{n}$ is $|\gamma_{n}|=2\sqrt{\rho^{2}-R_{n}^{2}}$.
Therefore,
\[
	E\Psi (\mathcal{L}\cap B (0,\varrho))=\kappa
	\int_{r=-\varrho }^{\varrho} 2\sqrt{\rho^{2}-r^{2}} \,dr=\kappa \pi \varrho^{2}.
\]

\end{proof}

\begin{proof}
  [Proof of Proposition~\ref{proposition:ergodicity-plp}] Let
  $\mathcal{L}$ be the Poisson line process with intensity $\kappa$,
  and denote by $\tau_{z}$ the translation by $z\in \zz{R}^{2}$.  It
  suffices to prove that for any two bounded, continuous functions
  $f,g:\mathcal{C}\rightarrow \zz{R}$,
\begin{equation}\label{eq:mixing-plp}
	\lim_{|z| \rightarrow \infty}Ef (\mathcal{L})g
	(\tau_{z}\mathcal{L})=
	Ef (\mathcal{L})Eg (\mathcal{L}).
\end{equation}
Since the Poisson line process is rotationally invariant, it suffices
to consider only translations $\tau_{z}$ for $z= (x,0)$ on the
$x-$axis. Moreover, since continuous functions that depend only on the
restrictions of configurations to balls are dense in the space of all
bounded, continuous functions, it suffices to establish
\eqref{eq:mixing-plp} for functions $f,g$ that depend only on
configurational restrictions to the ball of radius $r>0$ centered at
the origin.

To prove \eqref{eq:mixing-plp}, we will show that on some probability
space there are Poisson line processes
$\mathcal{L},\mathcal{L}',\mathcal{L}''$, each with intensity $\kappa$,
such that 
\begin{itemize}
\item [(a)] the line processes $\mathcal{L}'$ and $\mathcal{L}''$ are independent;
\item [(b)]  $f (\mathcal{L})=f (\mathcal{L}')$ with probability one; and
\item [(c)] $g (\tau_{z}\mathcal{L})=g (\tau_{z}\mathcal{L}'')$ with
probability $\rightarrow 1$ as $|z| \rightarrow \infty$. 
\end{itemize}
It will then follow, by  translation invariance, that
\begin{align*}
	|Ef (\mathcal{L})g(\tau_{z}\mathcal{L})-Ef (\mathcal{L})Eg
	(\mathcal{L})|&=|Ef (\mathcal{L})g(\tau_{z}\mathcal{L})-Ef (\mathcal{L}')g
	(\mathcal{L}'')|\\
	&=|Ef (\mathcal{L})g(\tau_{z}\mathcal{L})-Ef (\mathcal{L}')g
	(\tau_{z}\mathcal{L}'')|\\
	&\leq 2\xnorm{f}_{\infty}\xnorm{g}_{\infty} P\{g
	(\tau_{z}\mathcal{L})\not =g (\tau_{z}\mathcal{L}'') \}
	\longrightarrow 0.
\end{align*}
The line processes $\mathcal{L},\mathcal{L}',\mathcal{L}''$ can be
built on any probability space that supports independent Poisson point
processes $\{R'_{n} \}_{n\in \zz{Z}}$ and $\{R''_{n} \}_{n\in \zz{Z}}$
on $\zz{R}$ of intensity $\kappa$, and independent sequences
$\{\Theta '_{n}\}_{n\in \zz{Z}}$ and $\{\Theta ''_{n} \}_{n\in
\zz{Z}}$ of random variables uniformly distributed on the interval
$[-\pi ,\pi]$. Let $\mathcal{L}'$ be the line process obtained by
using the ``standard construction'' (that is, the construction
explained in Definition~\ref{definition:plp}) with the point process
$\{R'_{n} \}_{n\in \zz{Z}}$ and the accompanying uniform random
variables $\{\Theta '_{n}\}_{n\in \zz{Z}}$, and let $\mathcal{L}''$ be
the line process obtained by the standard construction using the point
process $\{R''_{n} \}_{n\in \zz{Z}}$ and the random variables
$\{\Theta ''_{n}\}_{n\in \zz{Z}}$.  Clearly, $\mathcal{L}'$ and
$\mathcal{L}''$ are independent.

The line process $\mathcal{L}$ is constructed by splicing the marked
Poisson point processes $\mathcal{R}'=\{(R'_{n} ,\Theta'_{n})\}_{n\in
\zz{Z}}$ and $\mathcal{R}''=\{(R''_{n} ,\Theta''_{n}) \}_{n\in
\zz{Z}}$ as follows: in the interval $(-r,r)$, use the marked points
of $\{(R'_{n} ,\Theta'_{n})\}_{n\in \zz{Z}}$; but in $\zz{R}\setminus
(-r,r)$, use the marked points of $\{(R''_{n} ,\Theta''_{n}) \}_{n\in
\zz{Z}}$. Thus, the resulting marked point process
$\mathcal{R}=\{(R_{n},\Theta_{n}) \}_{n\in \zz{Z}}$ consists of (i)
all pairs $(R'_{n},\Theta '_{n})$ such that $-r<R'_{n}<r$, and (ii)
all pairs $(R''_{n},\Theta ''_{n})$ such that $R''_{n}\not \in
(-r,r)$. By standard results in the elementary theory of Poisson
processes, the marked point process $\mathcal{R}$ has the same
distribution as $\mathcal{R}'$ and $\mathcal{R}''$, in particular,
$\{R_{n} \}_{n\in \zz{Z}}$ is a rate-$\kappa$ Poisson point process
on $\zz{R}$, and the random variables $\{\Theta_{n} \}_{n\in \zz{Z}}$
are independent and uniformly distributed on $[-\pi ,\pi]$. Let
$\mathcal{L}$ be the Poisson line process constructed using
$\mathcal{R}$.

It remains to show that the Poisson line processes
$\mathcal{L},\mathcal{L}',\mathcal{L}''$ satisfy properties (b) and
(c) above. Observe first that in the standard construction
(Definition~\ref{definition:plp}), only those pairs
$(R_{n},\Theta_{n})$ such that $R_{n}\in (-r,r)$ will produce lines
that intersect the ball $B (0,r)$ of radius $r$ centered at the
origin. Consequently, the restrictions of $\mathcal{L}$ and
$\mathcal{L}'$ to $B (0,r)$ are equal; since $f$ depends only on the
configuration in $B (0,r)$, it follows that $f (\mathcal{L})=f
(\mathcal{L}')$. 

Next, consider the configurational restrictions of $\mathcal{L}$ and
$\mathcal{L}''$ to the ball $B ((x,0),r)$ for $x\gg 2r$. In the
standard construction, a pair $(R_{n},\Theta_{n})$ such that $R_{n}\in
(-r,r)$ will produce a line of $\mathcal{L}$ that intersects $B
((x,0),r)$ only if $|\tan \Theta_{n}|\leq r/ (x-2r)$. The probability
that there is such a pair, in either $\mathcal{R}$ or $\mathcal{R}''$,
tends to $0$ as $x \rightarrow \infty$; hence, with probability
$\rightarrow 1$, the restrictions of $\mathcal{L}$ and $\mathcal{L}''$
agree in $B ((x,0),r)$, and on this event  $g (\mathcal{L})=g
(\mathcal{L}'')$. 
\end{proof}

\begin{proof}[Proof of Corollary~\ref{corollary:poly-freqs}]
The number of lines in a Poisson line process $\mathcal{L}$
that intersect a given line segment of length $m$ has the Poisson
distribution with mean $C{\kappa}m$, where $C$ is a finite positive
constant not depending on either $m$ or $\kappa$. Consequently,
the probability that the number of polygons in the induced
tessellation of the plane  intersecting one of the four sides of
$[-n,n]^{2}$ exceeds $n^{3/2}$ is exponentially small in $n$.

Given a line configuration $\mathcal{L}$, let $1/f (\mathcal{L})$ be
the area of the polygon containing the origin in the induced
tessellation. (This is well-defined and positive with probability
$1$.) Let $A^{-}_{n}$ be the union of all polygons of the tessellation
that lie entirely in the open square $(-n,n)^{2}$, and let $A^{+}_{n}$
be the union of the polygons that intersect $[-n,n]^{2}$. Then
\[
	\int_{A^{-}_{n}} f (\tau_{z}\mathcal{L}) \,dz
	\quad \text{and} \quad \int_{A^{+}_{n}} f (\tau_{z}\mathcal{L}) \,dz
\]
count the number of polygons in $A^{-}_{n}$ and $A^{+}_{n}$,
respectively; since the difference between these is less than
$n^{3/2}$, except with exponentially small probability, it follows
that except with small probability 
\[
	\left\lvert F_{n}- \int_{[-n,n]^{2}} f (\tau_{z}\mathcal{L})
	\,dz\right\rvert \leq n^{3/2}.
\]
Hence, by the multi-parameter ergodic theorem (see, for
example,~\cite{Calderon}), $F_{n}/n^{2}\rightarrow Ef (\mathcal{L})$
almost surely.

The proof of the assertion regarding empirical frequencies of $k-$
gons is similar. If  $G_{k}$ is the  event that the polygon containing
the origin is a $k-$gon, then  the total number of $k-$gons in the
region $A^{\pm}_{n}$ is
\[
	\int_{A^{\pm}_{n}} (f\mathbf{1}_{G_{k}}) (\tau_{z}\mathcal{L})\,dz.
\]
Hence,  the
ergodic theorem  implies that the number of $k-$gons divided by
$n^{2}$ converges to $E (f\textbf{1}_{G_{k}} (\mathcal{L}))$, and it
follows that the fraction of $k-$gons converges to 
\[
	\phi _{k}=\frac{E (f\textbf{1}_{G_{k}} (\mathcal{L}))}{Ef (\mathcal{L})}.
\]

Now consider the number of vertices $V_{n}$. Because there is
probability $0$ that three distinct lines of a Poisson line process
meet at a point, all interior vertices are shared by exactly 4 edges,
and each edge is incident to two vertices; thus, since the number of
vertices on the boundary of the square is $O (n^{3/2})$, we have
$\mathcal{E}_{n}=2V_{n}+O (n^{3/2})$.  By Euler's formula,
$V_{n}-\mathcal{E}_{n}+F_{n}=1$, so $V_{n}=F_{n}+O (n^{3/2})$; hence,
\[
	\lim_{n \rightarrow \infty }V_{n}/n^{2} = \lim_{n \rightarrow
	\infty }N_{n}/n^{2} .
\]
The value of the limit is determined by
Corollary~\ref{corollary:ppIntensity},
which implies that $EV_{n}=4\kappa^{2}n^{2}/\pi .$
\end{proof}

\bibliographystyle{plain}
\bibliography{mainbib}

\end{document}